\documentclass[12pt,twoside,a4paper]{amsart}
\usepackage{amssymb}
\usepackage{latexsym}
\usepackage{amsmath}
\usepackage{euscript}
\usepackage{graphics}

   \def\sH{{\mathfrak H}}   
      
   \def\sN{{\mathfrak N}}

      \def\dC{{\mathbb C}}

   \def\dN{{\mathbb N}}   
      \def\dR{{\mathbb R}}

\def\cD{{\mathcal D}}      \def\cF{{\mathcal F}}
   \def\cH{{\mathcal H}}   
      \def\cL{{\mathcal L}}
   \def\cN{{\mathcal N}}   \def\cO{{\mathcal O}}
\def\cP{{\mathcal P}}   \def\cQ{{\mathcal Q}}   
      \def\cU{{\mathcal U}}
      \def\cX{{\mathcal X}}
   \def\cZ{{\mathcal Z}}

\def\wt#1{{{\widetilde #1} }}
\def\wh#1{{{\widehat #1} }}

\def\bm\chi{\mbox{\boldmath$\chi$}}

\def\ran{{\rm ran\,}}

\def\dom{{\rm dom\,}}

\def\col{{\rm col\,}}
\def\dim{{\rm dim\,}}

\def\rank{{\rm rank\,}}

\let\xker=\ker \def\ker{{\xker\,}}

\def\cmr{{\dC \setminus \dR}}

\newtheorem{theorem}{Theorem}[section]
\newtheorem{proposition}[theorem]{Proposition}
\newtheorem{corollary}[theorem]{Corollary}
\newtheorem{lemma}[theorem]{Lemma}

\theoremstyle{definition}

\newtheorem{remark}[theorem]{Remark}
\newtheorem{definition}[theorem]{Definition}

\numberwithin{equation}{section}

\begin{document}

\title[Abstract interpolation problem]
{Abstract interpolation problem in Nevanlinna classes}
\author{Vladimir Derkach}

\address{Department of Mathematics \\
Donetsk National University \\
Universitetskaya str. 24 \\
83055 Donetsk \\
Ukraine} \email{derkach.v@gmail.com}

\date{\today}
\thanks{This research has been done partially while the author was visiting the Department of Mathematics of
Weizmann Institute of Science as a Weston Visiting Scholar }

\begin{abstract} The abstract interpolation  problem (AIP) in the Schur class was posed
V.~Katznelson, A. Kheifets and P. Yuditskii in 1987 as an extension of the
V.P.~Potapov's approach to interpolation problems. In the present paper an analog of the
AIP for Nevanlinna classes is considered. The description of solutions of the AIP  is
reduced to the description of L-resolvents of some model symmetric operator associated
with the AIP. The latter description is obtained by using the M.G.~Krein's theory of
L-resolvent matrices. Both regular and singular cases of the AIP are treated. The
results are illustrated by the following examples: bitangential interpolation problem,
full and truncated moment problems. It is shown that each of these problems can be
included into the general scheme of the AIP.
\end{abstract}

\maketitle
\section{Introduction}
The abstract interpolation  problem (AIP) was posed by V.~Katz\-nelson, A. Kheifets and
P. Yuditskii in~\cite{KKhYu} as an extension of the V.P.~Potapov's approach to
interpolation problems~\cite{KP}. In a sense the problem consists in contractive
"embedding" of some partial isometry $V$ acting in a structured Hilbert space
$\cH\oplus\cL$ into a model unitary operator acting in a space $\cH^w\oplus\cL$, where
$\cH^w$ is the de Branges-Rovnyak space corresponding to an operator valued function
(ovf) $w$ from the Schur class $S(\cL)$. A description of the set of all ovf $w\in
S(\cL)$ for which such an "embedding" is possible were reduced in~\cite{KKhYu} to the
description of all
scattering matrices of unitary extensions of a given partial isometry $V$ (\cite{ArGr}). 
In a number of papers it was shown that many problems of analysis, such that the
bitangential interpolation problem \cite{KhYu94}, moment problem~\cite{Kh96}, lifting
problem~\cite{Ku96}, and others can be included into the general scheme of AIP.

In the present paper we consider a parallel version of AIP for Nevanlinna class. The
class $N[\cL]$ consists of all ovf on $\dC_+\cup\dC_-$ with values in the set $[\cL]$ of
bounded linear operators in $\cL$ such that $m(\bar\lambda)=m(\lambda)^*$ and the kernel
\begin{equation}\label{eq:0.1}
    {\mathsf N}_\omega^m(\lambda)=\frac{m(\lambda)-m(\omega)^*}{\lambda-\bar\omega}
\end{equation}
is nonnegative on $\dC_+$. Then the kernel ${\mathsf N}_\omega^m(\lambda)$ is also
nonnegative on $\dC_+\cup\dC_-$.

In introduction we restrict ourselves to the case when $\dim\cL<\infty$ and identify
$\cL$ with the space $\dC^d$ where the standard basis is chosen. Then every ovf $m\in
N^{d\times d}:=N[\dC^d]$ can be considered as a $d\times d$ matrix valued function
(mvf). Let $\cH(m)$ be the reproducing kernel Hilbert space (see~\cite{Br77},
\cite{AG04}) which is characterized by the properties:
\begin{enumerate}
\item[(1)] ${\mathsf N}^{m}_\omega(\lambda)u\in\cH(m)$ for all $\omega\in\dC\setminus\dR$
and $u\in\dC^d$;
\item[(2)] for every $f\in\cH(m)$  the following identity holds
\begin{equation}\label{eq:0.2}
    \left\langle f(\cdot),{\mathsf
    N}^{m}_\omega(\lambda)u\right\rangle_{\cH({m})}=u^*f(\omega),\quad
\omega\in\dC\setminus\dR, u\in\dC^d.
\end{equation}
\end{enumerate}
The AIP in the class $N^{d\times d}$ can be formulated as follows.

Let $\cX$ be a complex linear space, let $B_1$, $B_2$ be linear operators in $\cX$, let
$C_1$, $C_2$ be linear operators from $\cX$ to $\dC^d$, and let $K$ be a nonegative
sesquilinear form on $\cX$ which satisfies the following identity
\begin{enumerate}
\item[(A1)]
$K(B_2h,B_1g)-K(B_1h,B_2g)=(C_1h,C_2g)_{\dC^d}-(C_2h,C_1g)_{\dC^d}$
\end{enumerate}
for every $h,g\in\ \cX$. Consider the following

\noindent {\bf Problem $AIP(B_1, B_2, C_1, C_2, K)$}. Let the data set $( B_1, B_2, C_1,
C_2, K)$ satisfies the assumption $(A1)$. Find a mvf $ m \in {N}^{d\times d}$ such that
for some linear mapping $F:X\to\cH(m)$ the following conditions hold:
\begin{enumerate}
\item[(C1)]
$(F{B_2h})(\lambda)-\lambda(F {B_1h})(\lambda)
=\left[\begin{array}{cc}I&-m(\lambda)\end{array}\right]\left[%
\begin{array}{c}
  C_1 h \\
C_2 h \\
\end{array}%
\right]$;
\item[(C2)] $\|F h\|_{\cH(m)}^2\leq K(h,h)$
\end{enumerate}
 for all $h\in \cX$.

Clearly $\ker K=\{h\in \cX:\,K(h,h)=0\}$ is a linear subspace of $\cX$.
Let $\cH$ be the completion of the factor-space $\wh{\cX}=\cX/\ker K$ endowed with the
scalar product
\begin{equation}\label{InnerP}
(\wh h,\wh g)_{\cH}=K(h,g), \quad \wh h=h+\ker K,\,\wh g=g+\ker K,\,\,h, g\in  \cX.
\end{equation}
It follows from (A1) that the linear relation
\[
    \wh{A}=\left\{\left\{%
\left[\begin{array}{c}
  \widehat{B_1h} \\
  C_1 h \\
\end{array}%
\right],\left[%
\begin{array}{c}
  \widehat{B_2h} \\
  C_2 h \\
\end{array}%
\right]\right\}: h\in\cX\right\}
\]
is symmetric in $\cH\oplus\dC^d$. The main result of the paper is the following
description of all the AIP solutions $m$.

\noindent{\bf Theorem 1.}\label{ThA} {\it Let the data set $(B_1,B_2,C_1,C_2,K)$
satisfies the assumption $(A1)$ and let $\ran C_2=\cL=\dC^d$. Then the Problem
$AIP(B_1,B_2,C_1,C_2,K)$ is solvable and the set of its solutions is parametrized by the
formula
\begin{equation}\label{eq:0.4}
\left[%
\begin{array}{c}
  m(\lambda) \\
  I_d \\
\end{array}%
\right]=\left[%
\begin{array}{cc}
  I_d & 0 \\
  \lambda & I_d \\
\end{array}%
\right]\left[%
\begin{array}{c}
  P_{\cL}(\wt{A}-\lambda)^{-1}I_{\cL} \\
      I_{\cL} \\
\end{array}%
\right](I_d+\lambda P_{\cL}(\wt{A}-\lambda)^{-1}I_{\cL})^{-1},
\end{equation}
where $\wt{A}$ ranges over the set of all selfadjoint extensions of $\wh{A}$ with the
exit in a Hilbert space $\wt{\cH}\oplus\cL\supset\cH\oplus\cL$. The corresponding linear
mapping $F:X\to\cH(\varphi,\psi)$ is given by
\begin{equation}\label{eq:0.51}
    (Fh)(\lambda)=P_\cL(\wt A-\lambda)^{-1}\wh h,\quad h\in X.
\end{equation}
}

Due to Theorem~1 the description of all solutions $m$ of the $AIP(B_1,B_2,C_1,C_2,K)$ is
reduced to the problem of description of all $\cL$-resolvents of the linear relation
$\wh A$. The latter description has been obtained by M.G. Kre\u{\i}n in~\cite{Kr49} (see
also~\cite{KrSa}). In order to apply this theory to the linear relation $\wh A$ we will
impose additional assumptions on the data set $(B_1$, $B_2$, $C_1$, $C_2$, $K)$.
\begin{enumerate}
\item[(A2)] $\dim\ker K<\infty$ and $\cX$ admits the representation
\begin{equation}\label{eq:0.61}
\cX=\cX_0\dotplus\ker K,
\end{equation}
such that $B_j\cX_0\subseteq\cX_0$ $( j=1,2)$.
\item[(A3)] $B_2=I_\cX$ and the operators $B_1|_{\cX_0}:\cX_0\subset\cH\to\cH$,
$C_1|_{\cX_0}$, $C_2|_{\cX_0}:\cX_0\subset\cH\to\cL$ are bounded.
\end{enumerate}
The continuations of the operators $B_1|_{\cX_0}$, $C_1|_{\cX_0}$, $C_2|_{\cX_0}$ will
be denoted by $\wt B_1\in[\cH]$, $\wt C_1$, $\wt C_2\in[\cH,\cL]$.

Denote by $\wt N^{d\times d}$ the set of Nevanlinna pairs $\{p,q\}$ of $d\times d$ mvfs
$p(\cdot)$, $q(\cdot)$ holomorphic on $\cmr$ such that:
\begin{enumerate}
\def\labelenumi{\rm (\roman{enumi})}
\item the kernel
$
 {\sf N}_\omega^{p,q} (\lambda)
  ={\displaystyle\frac{q(\bar\lambda)^*p(\bar\omega) -
   p(\bar\lambda)^*q(\bar\omega)}{\lambda-\bar\omega}}
$ is nonnegative on $\dC_+$; \item $q(\bar \lambda)^*p(\lambda)-p(\bar
\lambda)^*q(\lambda)=0$, $\lambda \in \cmr$; \item $0\in\rho(p(\lambda)-\lambda
q(\lambda))$, $\lambda \in\dC_\pm$.
\end{enumerate}
In the regular case ($\ker K=\{0\}$)
the set of $\cL$-resolvents $P_\cL(\wt A-\lambda)^{-1}|_\cL$ of $\wh A$ can be described
by the formula (see~\cite{Kr1}, \cite{KrSa})
\begin{equation}\label{eq:0.5}
  P_{\cL}(\wt{A}-\lambda)^{-1}I_{\cL} =(\wh w_{11}(\lambda)q(\lambda)+\wh w_{12}(\lambda)p(\lambda))(\wh
w_{21}(\lambda)q(\lambda)+\wh w_{22}(\lambda)p(\lambda))^{-1},
\end{equation}
where $\{p,q\}\in\wt N^{d\times d}$ 
and $\wh W=[ \wh w_{ij}(\lambda)]_{i,j=1}^2$
is an $\cL-$resolvent matrix  of $\wh A$ which can be calculated explicitly in terms of
the data set (see~\eqref{eq:4.81}). Combining~\eqref{eq:0.4} and~\eqref{eq:0.5} one
obtains the following

\noindent
{\bf Theorem 2.}\label{Lres} {\it Let the AIP data set $(B_1$, $B_2$, $C_1$,
$C_2$, $K)$ satisfy $(A1)$, $(A3)$,  $\ker K=\{0\}$ and $\ran C_2=\cL=\dC^d$
and let 
\begin{equation}\label{eq:0.6}
    \Theta(\lambda)=\begin{bmatrix} \theta_{11}(\lambda) &
\theta_{12}(\lambda)\\\theta_{21}(\lambda) &
\theta_{22}(\lambda)\end{bmatrix}=I_{\cL\oplus\cL}-\lambda \begin{bmatrix} \wt C_1\\\wt
C_2\end{bmatrix}
    (I_\cH-\lambda \wt B_1)^{-1}\begin{bmatrix}  -\wt C_2^+ &\ \wt C_1^+\end{bmatrix}.
\end{equation}
 Then the formula
\begin{equation}\label{eq:0.7}
    m(\lambda)=(\theta_{11}(\lambda)q(\lambda)+\theta_{12}(\lambda)p(\lambda))
    (\theta_{21}(\lambda)q(\lambda)+\theta_{22}(\lambda)p(\lambda))^{-1},
\end{equation}
 establishes the 1-1 correspondence between the set of all  solutions
$m$ of the $AIP$ and the set of all equivalence classes of Nevanlinna pairs
$\{p,q\}\in\wt N(\cL)$. }

The operators $\wt C_1^+$, $\wt C_2^+$ in~\eqref{eq:0.6} are adjoint operators to $\wt
C_1$, $\wt C_2:\cH\to\cL$, that is
\[
(\wt C_j^+u,h)_\cH=(u,\wt C_jh)_\cL\quad (j=1,2,h\in\cH,u\in\cL).
\]
Let the matrix $J\in\dC^{2d\times 2d}$ be given by
\[
J=\begin{bmatrix} 0 & -iI_d\\
                    iI_d & 0
\end{bmatrix}
\]
The mvf $\Theta(\lambda)$ belongs to the Potapov class (see~\cite{Pot}), i.e.
$\Theta(\lambda)$ has the following $J$-property
\[
\frac{J-W(\lambda)JW(\lambda)^*}{\lambda-\bar\lambda}\ge 0\quad \mbox{ for all }
\lambda\in\dC\setminus\dR.
\]

 In the singular case the
formula~\eqref{eq:0.5} gives a description of all $\cL$-resolvents of the linear
relation
\[
A_0=\left\{\left\{\left[%
\begin{array}{c}
  B_1 h \\
  C_1 h \\
\end{array}%
\right],\left[%
\begin{array}{c}
  B_2 h \\
  C_2 h \\
\end{array}%
\right]\right\}:h\in\cX_0\right\}.
\]
To obtain a description of all $\cL$-resolvents of $\wh A(\supset A_0)$ one should
consider in~\eqref{eq:0.5} only those Nevanlinna pairs $\{p,q\}\in\wt N(\cL)$ for which
$\wt A\supset\wh A(\supset A_0)$. We show that after the replacement in~\eqref{eq:0.7}
of $\Theta(\lambda)$ by the
 $\Theta(\lambda)V$, where $V$ is an appropriate $J$-unitary
matrix, the formula~\eqref{eq:0.7} gives a description of all the solutions of the AIP
when $\{p,q\}$ ranges over the set of all Nevanlinna pairs of the form
\begin{equation}\label{eq:pq}
p(\lambda)=\begin{bmatrix} \wt 0_\nu & 0\\ 0 & p_1(\lambda)\end{bmatrix},\quad
q(\lambda)=\begin{bmatrix} \wt I_\nu & 0\\ 0 & q_1(\lambda)\end{bmatrix},\quad
\{p_1,q_1\}\in\wt N^{d-\nu},
\end{equation}
where $\nu=\dim C\ker K$.

All these results are formulated in the paper in a more general situation, when the mvf
$m$ is replaced by a Nevanlinna pair $\{\varphi,\psi\}$. Moreover, we do not suppose, in
general, that $\dim\cL<\infty$.

The paper is organized as follows. In Section 2 we recall the definition of the class
$\wt N(\cL)$ of Nevanlinna pairs from~\cite{ABDS93},~\cite{DHMS}. To each selfadjoint
linear relation $\wt A$ and a scale spaces $\cL$ we associate a Nevanlinna pair
$\{\varphi,\psi\}$ by the formula
\begin{equation}\label{eq:0.11}
   \psi(\lambda):=P_\cL(\wt A-\lambda)^{-1}|_{\cL},\quad
   \varphi(\lambda):=I_\cL+\lambda P_\cL(\wt A-\lambda)^{-1}|_{\cL},
    \lambda\in\rho(\wt A).
\end{equation}
Conversely, given a Nevanlinna pair $\{\varphi,\psi\}$ normalized  by the condition
$p(\lambda)-\lambda q(\lambda)=I_d$ we construct a functional model for a selfadjoint
linear relation $\wt A(\varphi,\psi)$, such that the pair $\{\varphi,\psi\}$ is related
to $\wt A(\varphi,\psi)$ via~\eqref{eq:0.11}. In the case when the pair
$\{\varphi,\psi\}$ is equivalent to a pair $\{I_d,m\}$ with $m\in N[\cL]$, functional
model $A(m)$ for this symmetric operator was given in~\cite{ABDS93} (see
also~\cite{DM95}). Conditions when the model $A(\varphi,\psi)$ is reduced to $A(m)$ are
discussed. In Sections 3 and 4 we formulate the AIP in the classes $N[\cL]$ and $\wt
N(\cL)$ and give a complete description of its solutions under some additional
restrictions on the data set both in the regular and singular cases. The results of the
paper are illustrated in Section 5 with an example of bitangential interpolation
problems in the classes $N^{d\times d}$ and $\wt N^{d\times d}$, reduced there to the
AIP with appropriately chosen data set. These problems have been studied earlier
in~\cite{Nud77}, \cite{ABDS93},  \cite{Kh1}, \cite{Dym89}, \cite{Dym01},
\cite{BolDym98}.

Mention, that the Arov and Grossman's description of scattering matrices of unitary
extensions of an isometry $V$ in~\cite{ArGr} used in the Schur type AIP is an analog of
M.G. Kre\u{\i}n's description~\eqref{eq:0.5} of $\cL$-resolvents of a symmetric
operator~\cite{Kr1}. One of the goals of this paper is the formulation of the AIP, where
the M.G.~Kre\u{\i}n's formula~\eqref{eq:0.5} works directly. In particular, we use the
example of the full moment problem to show that the reduction of this problem to the
Nevanlinna type AIP is more natural and simpler than that in~\cite{Kh96}, where the
reduction of the moment problem to the Schur type AIP was performed.

Another goal of the paper was to elaborate the operator approach to singular AIP. This
approach is illustrated with an example of singular truncated moment problem, where we
discuss the results of~\cite{Bol96} and explain them from our point of view.

The paper is dedicated to the centennial of M.G.~Kre\u{\i}n.

\section{Functional model of a selfadjoint linear relation}
\subsection{Nevanlinna pairs} Let $\cL$ be a Hilbert space.
\begin{definition}
\label{npair}  A pair $\{\Phi,\Psi\}$ of $[{\cL}]$-valued functions $\Phi(\cdot)$,
$\Psi(\cdot)$ holomorphic on $\cmr$ is said to be a \textit{ Nevanlinna pair} if:
\begin{enumerate}
\def\labelenumi{\rm (\roman{enumi})}
\item the kernel
\[
 {\sf N}_\omega^{\Phi,\Psi} (\lambda)
  =\frac{\Psi(\bar\lambda)^*\Phi(\bar\omega) -
   \Phi(\bar\lambda)^*\Psi(\bar\omega)}{\lambda-\bar\omega},
\quad
 \lambda,\omega \in \dC_+
\]
is nonnegative on $\dC_+$; \item $\Psi(\bar \lambda)^*\Phi(\lambda)-\Phi(\bar
\lambda)^*\Psi(\lambda)=0$, $\lambda \in \cmr$; \item
$0\in\rho(\Phi(\lambda)-\lambda\Psi(\lambda))$, $\lambda \in\dC_\pm$. \end{enumerate}
\end{definition}

Two Nevanlinna pairs $\{\Phi,\Psi\}$ and $\{\Phi_1,\Psi_1\}$ are said to be
\textit{equivalent}, if $\Phi_1 (\lambda)=\Phi(\lambda)\chi(\lambda)$ and
$\Psi_1(\lambda)=\Psi(\lambda)\chi(\lambda)$ for some operator function
$\chi(\cdot)\in[\cH]$, which is holomorphic and invertible on $\dC_+\cup\dC_-$. The set
of all equivalence classes of Nevanlinna pairs in $\cL$ will be denoted by $\wt N(\cL)$.
We will write, for short, $\{\Phi,\Psi\}\in \wt N(\cL)$ for Nevanlinna pair
$\{\Phi,\Psi\}$.

 A Nevanlinna pair $\{\Phi,\Psi\}$ will be said to be \textit{normalized} if
$\Phi(\lambda)-\lambda\Psi(\lambda)=I_\cH$. Clearly, every Nevanlinna pair
$\{\Phi,\Psi\}$ is equivalent to the unique normalized Nevanlinna pair
$\{\varphi,\psi\}$ given by
\begin{equation}\label{NormNP}
\varphi(\lambda)=\Phi(\lambda)(\Phi(\lambda)-\lambda\Psi(\lambda))^{-1},\quad
\psi(\lambda)=\Psi(\lambda)(\Phi(\lambda)-\lambda\Psi(\lambda))^{-1}.
\end{equation}
The set $\wt N(\cL)$ can be identified with the set of Nevanlinna families
(see~\cite{DHMS})
\begin{equation}\label{NevFam}
    \tau(\lambda)=\{\{\Phi(\lambda)u,
    \Psi(\lambda)u\}:\,u\in\cL\},\quad\{\Phi,\Psi\}\in\wt N(\cL).
\end{equation}
Define the class $ N(\cL)$ as the set of all Nevanlinna pairs $\{\Phi,\Psi\}\in \wt
N(\cL)$ such that $\ker\Phi(\lambda)=\{0\}$ for some (and hence for all)
$\lambda\in\cmr$. The set $N[\cL]$ can be embedded in $ N(\cL)$ via the mapping
\[
m\in N[\cL]\mapsto \{I_\cL,m\}\in N(\cL).
\]
\subsection{Nevanlinna pair corresponding to a selfadjoint linear relation and a
scale.} Let $\sH$, $\cL$ be Hilbert spaces, let $\wt A$ be a selfadjoint linear relation
in $\sH\oplus\cL$ and let $P_\cL$ be the orthogonal projection onto the scale space
$\cL$. Define the operator valued functions
\begin{equation}\label{eq:1.1}
   \psi(\lambda):=P_\cL(\wt A-\lambda)^{-1}|_{\cL},\quad
   \varphi(\lambda):=I_\cL+\lambda P_\cL(\wt A-\lambda)^{-1}|_{\cL},
    \lambda\in\rho(\wt A).
\end{equation}

\begin{proposition}
\label{pr:1.1} The pair of operator valued functions
$\{\varphi,\psi\}$, associated with a selfadjoint linear relation
$\wt A$ and the scale $\cL$ is a normalized Nevanlinna pair.
\end{proposition}
\begin{proof}
Consider the kernel
\begin{equation}\label{eq:1.3}
{\mathsf N}_\omega^{{\varphi}{\psi}}(\lambda)=\frac{\psi(\bar\lambda)^*\phi(\bar\omega)
-\phi(\bar\lambda)^*\psi(\bar\omega)}{\lambda-\bar\omega},
\quad\lambda,\omega\in\dC_+\cup\dC_-.
\end{equation}
It follows from~\eqref{eq:1.1}-\eqref{eq:1.3} that
\begin{equation}\label{eq:1.31}
\begin{split}
{\mathsf
N}_\omega^{{\varphi}{\psi}}(\lambda)&=\frac{\psi(\lambda)-\psi(\omega)^*}{\lambda-\bar\omega}
-\psi(\lambda)\psi(\omega)^*\\
&=P_\cL\frac{R_\lambda-R_{\bar\omega}}{\lambda-\bar\omega}|_\cL-
P_\cL R_\lambda P_\cL R_{\bar\omega}|_\cL\\
&=P_\cL R_\lambda P_\cH R_{\bar\omega}|_\cL
\end{split}
\end{equation}
and hence the kernel ${\mathsf N}_\omega^{{\varphi}{\psi}}(\lambda)$ is nonnegative.

The property (ii) is easily checked, $\varphi(\lambda)-\lambda\psi(\lambda)=I_\cL$
and, hence, the pair 
$\{\phi,\psi\}$  is a normalized Nevanlinna pair.
\end{proof}
\begin{definition}\label{def:1.1}
The pair of operator valued functions determined by~\eqref{eq:1.1} will be called the
Nevanlinna pair corresponding to the selfadjoint linear relation $\wt A$ and the scale
$\cL$.
\end{definition}
\begin{remark}\label{rem:1.1}
Definition~\ref{def:1.1} is inspired by the notion of the Weyl family of a symmetric
operator corresponding to a boundary relation, see~\cite{DHMS06}. Namely, the Nevanlinna
pair $\{\varphi,\psi\}$ determined by~\eqref{eq:1.1} generates via~\eqref{NevFam} the
Weyl family of the symmetric linear relation $S=\wt A\cap(\cH\oplus\cH)$, corresponding
to the boundary relation
\[
\Gamma=
 \left\{\,
 \left\{ \begin{bmatrix} f \\ f' \end{bmatrix},
         \begin{bmatrix} h' \\ h \end{bmatrix}
 \right\} :\,
 \left\{ \begin{bmatrix} f \\ h \end{bmatrix},
         \begin{bmatrix} f' \\ h' \end{bmatrix}
 \right\}
 \in \wt A
 \,\right\}.
\]
The proof of Proposition~\ref{pr:1.1} is contained in~\cite[Theorem~3.9]{DHMS06}.
Moreover, it is shown in~\cite{DHMS06} that the converse is also true, every Nevanlinna
family generates via~\eqref{NevFam} the Weyl family of a symmetric linear operator $S$.
In the case when the pair $\{\varphi,\psi\}$ is equivalent to a pair $\{I_\cL,m\}$ with
$m\in N[\cL]$, functional model $A(m)$ for this symmetric operator $S$ was given
in~\cite{ABDS93} (see also~\cite{DM95}).
\end{remark}

In the following theorem we give another functional model of a selfadjoint linear
relation $\wt A$ recovered from a Nevanlinna pair. Consider the reproducing kernel
Hilbert space $\cH(\Phi,\Psi)$, which is characterized by the properties:
\begin{enumerate}
\item[(1)] ${\mathsf N}^{\Phi\Psi}_\omega(\lambda)u\in\cH(\Phi,\Psi)$ for all $\omega\in\dC\setminus\dR$
and $u\in\cL$;
\item[(2)] for every $f\in\cH(\Phi,\Psi)$  the following identity holds
\begin{equation}\label{eq:1.4}
    \left\langle f(\cdot),{\mathsf
    N}^{\Phi\Psi}_\omega(\lambda)u\right\rangle_{\cH({\Phi,\Psi})}=(f(\omega),u)_\cL,\quad
\omega\in\dC\setminus\dR, u\in\cL.
\end{equation}
\end{enumerate}
It follows from~\eqref{eq:1.4} that the evaluation operator $E(\lambda):f\mapsto
f(\lambda)$ $(f\in\cH({\Phi,\Psi}))$ is a bounded operator from $\cH(\Phi,\Psi)$ to
$\cL$.
\begin{theorem}\label{thm:1.2}
Let $\{\Phi,\Psi\} \in \wt{N}(\cL)$. Then the linear relation
\begin{equation}\label{eq:1.5}
    \wt A({\Phi,\Psi})=\left\{\left\{\left[%
\begin{array}{c}
  f \\
  u \\
\end{array}%
\right],\left[%
\begin{array}{c}
  f' \\
  u' \\
\end{array}%
\right]\right\}:\begin{array}{c}
       f,f'\in \cH({\Phi,\Psi}),\,u,u'\in\cL, \\
                  f'(\lambda)-\lambda f(\lambda)=\Phi(\bar\lambda)^*u-\Psi(\bar\lambda)^*u' \\
                \end{array}  \right\}
\end{equation}
is a selfadjoint linear relation in $\cH(\Phi,\Psi)\oplus\cL$ and the normalized pair
$\{\varphi,\psi\}$ given by~\eqref{NormNP} is the Nevanlinna pair corresponding to $\wt
A(\Phi,\Psi)$ and $\cL$.
\end{theorem}
\begin{proof} {\it Step 1.} Let us show that $\wt A({\Phi,\Psi})$ contains
vectors of the form
\begin{equation}\label{eq:1.6}
\{ F_\omega u,F'_\omega u\}:=   \left\{\left[%
\begin{array}{c}
  {\mathsf N}_\omega(\cdot)u \\
  \Psi(\bar\omega) \\
\end{array}%
\right],\left[%
\begin{array}{c}
 \bar \omega{\mathsf N}_\omega(\cdot)u \\
  \Phi(\bar\omega)u \\
\end{array}%
\right]\right\},\quad u\in\cL,\,\,\omega\in\dC_+\cup\dC_-,
\end{equation}
where ${\mathsf N}_\omega(\cdot)={\mathsf N}^{\Phi\Psi}_\omega(\cdot)$ and the
restriction $\wt A'$ of $\wt A({\Phi,\Psi})$ to the span of vectors $\{ F_\omega
u,F'_\omega u\}$ is a symmetric linear relation.

Indeed, it follows from~\eqref{eq:1.5} and the equality
\[
(\bar \omega-\lambda){\mathsf N}_\omega(\lambda)u=\Phi(\bar\lambda)^*\Psi(\bar\omega) -
   \Psi(\bar\lambda)^*\Phi(\bar\omega)
\]
that $\{ F_\omega u,F'_\omega u\}\in\wt A({\Phi,\Psi})$.

For arbitrary $\omega_j\in\dC\setminus\dR$, $u_j\in\cL$ $(j=1,2)$
one obtains
\[
\begin{split}
\left\langle\bar\omega_1 \right.&\left.{\mathsf
N}_{\omega_1}(\cdot)u_1,{\mathsf
N}_{\omega_2}(\cdot)u_2\right\rangle_{\cH(\Phi,\Psi)}
-\left\langle{\mathsf N}_{\omega_1}(\cdot)u_1,\bar\omega_2{\mathsf
N}_{\omega_2}(\cdot)u_2\right\rangle_{\cH(\Phi,\Psi)}\\
&+(\Phi(\bar\omega_1)u_1,\Psi(\bar\omega_2)u_2)_\cL-(\Psi(\bar\omega_1)u_1,\Phi(\bar\omega_2)u_2)_\cL\\
&=(\bar\omega_1-\omega_2)({\mathsf N}_{\omega_1}(\omega_2)u_1,u_2)_\cL
-((\Phi(\bar\omega_2)^*\Psi(\bar\omega_1)-\Psi(\bar\omega_2)^*\Phi(\bar\omega_1))u_1,u_2)_\cL\\
&=0,
\end{split}
\]
therefore, $\wt A'$ is symmetric in $\cH(\Phi,\Psi)\oplus\cL$.

{\it Step 2.} Let us show that $\ran(\wt A'-\lambda)$ is dense in
$\cH(\Phi,\Psi)\oplus\cL$ for $\lambda\in\dC\setminus\dR$. Choose
the vector $\{ F_{\omega}u,F'_\omega u\}$ with $\omega=\bar\lambda$.
Then
\begin{equation}\label{eq:1.61}
  \left\{\left[%
\begin{array}{c}
  {\mathsf N}_{\bar\lambda}(\cdot)u \\
  \Psi(\lambda)u \\
\end{array}%
\right],\left[%
\begin{array}{c}
 0 \\
  \Phi(\lambda)u -\lambda \Psi(\lambda)u\\
\end{array}%
\right]\right\}\in\wt A'-\lambda.
\end{equation}
Since $\ran(\Phi(\lambda)u -\lambda \Psi(\lambda))=\cL$ one obtains
$0\oplus\cL\subset\ran(\wt A'-\lambda)$. Taking $\wh F_{\omega}u$ with
$\omega\ne\bar\lambda$ one obtains from~\eqref{eq:1.6}
\begin{equation}\label{eq:1.62}
    \left\{\left[%
\begin{array}{c}
  {\mathsf N}_{\omega}(\cdot)u \\
  \Psi(\bar\omega)u \\
\end{array}%
\right],\left[%
\begin{array}{c}
 (\bar\omega-\lambda){\mathsf N}_{\omega}(\cdot)u  \\
  \Phi(\bar\omega)u -\lambda \Psi(\bar\omega)u\\
\end{array}%
\right]\right\}\in\wt A'-\lambda
\end{equation}
and, hence, $\begin{bmatrix}{\mathsf N}_{\omega}(\cdot)u \\0
\end{bmatrix} \in\ran(\wt A'-\lambda)$ for all $\omega\ne\bar\lambda$.
Due to the properties 1) and 2) of $\cH(\Phi,\Psi)$ one obtains the
statement.

{\it Step 3.} Let us show that $\wt A({\Phi,\Psi})=(\wt A')^*$. Indeed, for every vector
\[
\wh F= \{F,F'\}=  \left\{\left[%
\begin{array}{c}
  f(\cdot) \\
  u \\
\end{array}%
\right],\left[%
\begin{array}{c}
 f'(\cdot) \\
 u' \\
\end{array}%
\right]\right\}\in\wt A({\Phi,\Psi}),\quad f,f'\in \cH(\Phi,\Psi),\,u,u'\in\cL,
\]
and $\omega\in\dC\setminus\dR$, $v\in\cL$ it follows
from~\eqref{eq:1.5} that
\[
\begin{split}
 \left\langle F',F_{\omega}v\right\rangle_{\cH(\Phi,\Psi)}-
  \left\langle F,F'_{\omega}v\right\rangle_{\cH(\Phi,\Psi)}&=
 \left\langle f',{\mathsf
N}_{\omega}(\cdot)v\right\rangle_{\cH(\Phi,\Psi)} -\left\langle
f,\bar\omega{\mathsf
N}_{\omega}(\cdot)v\right\rangle_{\cH(\Phi,\Psi)}\\
&+(u',\Psi(\bar\omega)v)_\cL-(u,\Phi(\bar\omega)v)_\cL\\
&=(f'(\omega)-\omega f(\omega)+\Psi(\bar\omega)^*u'-\Phi(\bar\omega)^*u,v)_\cL\\
&=0
\end{split}
\]
Hence $\wh F\in (\wt A')^*$ and $\wt A({\Phi,\Psi})\subset(\wt A')^*$.

Conversely, if
\[
\left\langle f',{\mathsf N}_{\omega}(\cdot)h\right\rangle_{\cH(\Phi,\Psi)} -\left\langle
f,\bar\omega{\mathsf
N}_{\omega}(\cdot)v\right\rangle_{\cH(\Phi,\Psi)}+(u',\Psi(\bar\omega)v)_\cL-(u,\Phi(\bar\omega)v)_\cL=0
\]
for some $f,f'\in \cH(\varphi,\psi)$, $u,u'\in\cL$ and all
$\omega\in\dC\setminus\dR$, $v\in\cL$, then
\[
f'(\omega)-\omega f(\omega) -(\Phi(\bar\omega)^*u-\Psi(\bar\omega)^*u')=0
\]
and, hence, $\wh F=\left\{\left[%
\begin{array}{c}
  f(\cdot) \\
  u \\
\end{array}%
\right],\left[%
\begin{array}{c}
 f'(\cdot) \\
 u' \\
\end{array}%
\right]\right\}\in\wt A({\Phi,\Psi})$. This proves that $(\wt A')^*\subset \wt
A({\Phi,\Psi})$.

{\it Step 4.} And finally, in view of~\eqref{eq:1.61} and the property (ii) of
Definition~\ref{npair} one obtains
\[
\begin{split}
\psi(\lambda)&=P_{\cL}(\wt
A({\Phi,\Psi})-\lambda)^{-1}|_\cL
=\Psi(\lambda)(\Phi(\lambda)-\lambda\Psi(\lambda))^{-1},\\
\varphi(\lambda)&=I_{\cL}+\lambda\psi(\lambda)
=\Phi(\lambda)(\Phi(\lambda)-\lambda\Psi(\lambda))^{-1}.
\end{split}
\]
Therefore, the pair  $\{\varphi,\psi\}$ is a normalized Nevanlinna pair  corresponding
to the linear relation $\wt A({\Phi,\Psi})$ and the scale $\cL$.
\end{proof}

\begin{remark}
In notations of~\cite{DHMS06} the pair $\{\Phi,\Psi\}$ generates via~\eqref{NevFam} the
Weyl family of the symmetric operator
\[
S(\Phi,\Psi)=\{\{f,f'\}:\,f,f'\in\cH(\Phi,\Psi),f'(\lambda)-\lambda f(\lambda)=0\}
\]
corresponding to the boundary relation
\[
\Gamma=\left\{\left\{\left[%
\begin{array}{c}
  f \\
  f' \\
\end{array}%
\right],\left[%
\begin{array}{c}
  u' \\
  u \\
\end{array}%
\right]\right\}:\begin{array}{c}
       f,f'\in \cH({\Phi,\Psi}),\,u,u'\in\cL, \\
                  f'(\lambda)-\lambda f(\lambda)=\Phi(\bar\lambda)^*u-\Psi(\bar\lambda)^*u' \\
                \end{array}  \right\}.
\]
\end{remark}
\begin{remark}
Mention that the linear space
\[
\sN_{\bar{\omega}}(T):=\{{\mathsf N}_{\omega}(\cdot)u:u\in\cL\}
\]
in general is not closed, since
\[
\begin{split}
({\mathsf N}_{\omega}(\cdot)u,{\mathsf
N}_{\omega}(\cdot)u)_{\cH(\varphi,\psi)}
&=({\mathsf N}_{\omega}(\omega)u,u)_{\cL}\\
&=\left(\frac{\Phi(\bar\omega)^*\Psi(\bar\omega)-\Psi(\bar\omega)^*\Phi(\omega)}{\omega-\bar{\omega}}
u,u\right)
\end{split}
\]
and the operator ${\mathsf N}_{\omega}(\omega)$ not necessarily is
boundedly invertible. If, however, $0\in \rho({\mathsf
N}_{\omega}(\omega))$ then ${\sN}_{\bar\omega}(T)$ is closed. Recall
that in this case $0\in\rho({\mathsf N}_{\lambda}(\lambda))$ for all
$\lambda\in\dC\backslash\dR$ and, hence, all the subspace
$\sN_{\lambda}(\bar\lambda)$ are closed.

Let, the ovf $\gamma(\lambda):\cL\rightarrow\cH$ be defined by
\begin{equation}\label{eq:1.7}
   \gamma(\lambda):=P_{\cH}(\wt{A}-\lambda)^{-1}|_{\cL}\quad (\lambda\in\rho(\wt A)).
\end{equation}
\end{remark}
\begin{proposition}\label{pr:1.3}
Let $\wt{A}$ be a selfadjoint linear relation in $\cH\oplus\cL$ and let
$\{\varphi,\psi\}$ be the normalized Nevanlinna pair given by~\eqref{eq:1.1}. Then the
following identity holds
\begin{equation}\label{eq:1.8}
 {\mathsf N}_{\omega}^{\varphi\psi}(\lambda)
=\gamma({\bar\lambda})^*\gamma({\bar\omega}).
\end{equation}
\end{proposition}
\begin{proof}
Indeed, it follows from~\eqref{eq:1.31} that the kernel ${\mathsf
N}_{\omega}^{\varphi\psi}(\lambda)$ for the normalized Nevanlinna pair takes the form
\[
 {\mathsf N}_{\omega}^{\varphi\psi}(\lambda)=(P_{\cL}R_{\lambda}P_{\cH})
(P_{\cH}R_{\bar{\omega}}|_{\cL})=\gamma({\bar\lambda})^*\gamma({\bar\omega}).
\]
\end{proof}
In general case one obtains
\[
\begin{split}
{\mathsf N}_{\omega}^{\Phi\Psi}(\lambda)
&=(\Phi(\bar\lambda)-\lambda\Psi(\bar\lambda))^*{\mathsf
N}_{\omega}^{{\varphi}{\psi}}(\lambda)
(\Phi(\bar\omega)-\bar\omega\Psi(\bar\omega))\\
&=(\Phi(\bar\lambda)-\lambda\Psi(\bar\lambda))^*\gamma(\bar{\lambda})^*\gamma(\bar{\omega})
(\Phi(\bar\omega)-\bar\omega\Psi(\bar\omega)).
\end{split}
\]
The following statement formulated in terms of boundary relations
can be found in~\cite[Lemma 4.1]{DHMS}
\begin{lemma}\label{lm:1.1}
Let $\wt{A}$ be a selfadjoint linear relation in $\cH\oplus\cL$, let $\{\varphi,\psi\}$
be the normalized Nevanlinna pair given by~\eqref{eq:1.1} and lel $\dim \cL<\infty$.
Then:
\begin{enumerate}
\item[(i)]
$\ker\psi(\lambda)={0}$ for $\lambda\in\dC\setminus\dR$ iff $P_\cL\,\dom\wt A$ is dense
in $\cL$;
\item[(ii)]
$\ker\varphi(\lambda)={0}$ for $\lambda\in\dC\setminus\dR$ iff $P_\cL\,\ran\wt A$ is
dense in $\cL$.
\end{enumerate}
\end{lemma}
\begin{proof}
Let us prove the first statement. The set $P_\cL\dom\wt A$ consists
of the vectors $u\in \cL$ such that
\[
\left\{\left[%
\begin{array}{c}
  f \\
  u \\
\end{array}%
\right],\left[%
\begin{array}{c}
  f' \\
  u' \\
\end{array}%
\right]\right\}\in \widetilde{A}\quad \mbox{for some}\quad f,f'\in
\cH,u'\in\cL.
\]
If there is a vector $v\in\cL$ such that $v\perp u$ for all $u\in
P_\cL\dom\wt A$ then
\[
\left\{\left[%
\begin{array}{c}
  0 \\
  0 \\
\end{array}%
\right],\left[%
\begin{array}{c}
  0 \\
  v \\
\end{array}%
\right]\right\}\in \widetilde{A},
\]
and then $\psi(\lambda)v=0$, $\varphi(\lambda)v=v$, due to~\eqref{eq:1.1}.

Conversely, if $\psi(\lambda)v=0$ for some $v\ne 0$, then in view of~\eqref{eq:1.1}
\[
\left\{\left[%
\begin{array}{c}
  0 \\
  0 \\
\end{array}%
\right],\left[%
\begin{array}{c}
  0 \\
  v \\
\end{array}%
\right]\right\}\in \wt{A},
\]
and hence $v\perp  P_\cL\dom\wt A$.

The proof of the second statement is similar.
\end{proof}
If the Nevanlinna pair $\{\varphi,\psi\}$ satisfies the first condition (i) in
Lemma~\ref{lm:1.1} then it is equivalent to a Nevanlinna function $m\in N(\cL)$. If,
additionally, $m(\lambda)$ takes values in $[\cL]$ for $\lambda\in\dC\setminus\dR$ then
the reproducing kernel Hilbert space $\cH(\varphi,\psi)$ is unitary equivalent to the
reproducing kernel Hilbert space $\cH(m)$ with the kernel ${\mathsf
N}_{\omega}^{m}(\lambda)$ (see~\eqref{eq:0.1}) via the mapping
\[
\cU:f\in\cH(m)\to(I-\lambda m(\lambda))^{-1}f(\lambda)\in\cH(\varphi,\psi).
\]
These spaces have been introduced by L.~de~Branges~\cite{Br77}. The following statement
can be derived from~\cite{AD84}, however we will give a proof for the convenience of the
reader.
\begin{lemma}\label{lm:2.}
Let $m\in N[\cL]$, and let $\dim\cL<\infty$.
Then:
\begin{enumerate}
\item[(i)]
$f(\lambda)=O(1)$ ($\lambda\widehat{\rightarrow}\infty$) for all
$f\in\cH(m)$;
\item[(ii)]
If, additionally, $m\in N_0[\cL]$ then $f(\lambda)=O(\frac{1}{\lambda})$
($\lambda\widehat{\rightarrow}\infty$) for all $f\in\cH(m)$.
\end{enumerate}
\end{lemma}
\begin{proof}
It follows from the reproducing kernel property and Schwartz
inequality that
\[
\begin{split}
|(f(\lambda),u)_\cL|
&=|\langle f(\cdot),{\mathsf N}_{\lambda}(\cdot)u\rangle_{\cH(m)}|\\
&\leq||f(\cdot)||_{\cH(m)}||{\mathsf N}_{\lambda}(\cdot) u||_{\cH(m)}\\
&=||f||_{\cH(m)}\left(\frac{\Im m(\lambda)}{\Im\lambda}u, u\right)^{1/2}=O(1)
\end{split}
\]
for all $f\in\cH(m)$ and $u\in\cL$. If, additionally, $m\in N_0[\cL]$, then
\[
{\displaystyle \left(\frac{\Im
m(\lambda)}{\Im\lambda}u,u\right)=\int\frac{d(\sigma(t)u,u)}{|t-\lambda|^2}=O(\frac{1}{\lambda^2})}
\]
for some finite measure $d\sigma$ and, hence,
\[
(f(\lambda), u)_\cL=O(\frac{1}{\lambda}).
\]
\end{proof}

\subsection{Generalized Fourier transform}
\begin{definition}
A selfadjoint linear relation $\wt{A}$ in $\cH\oplus\cL$ is said to
be $\cL$-minimal if
\begin{equation}\label{eq:2.1}
  \cH_0=\overline{\mbox{span}}\{P_{\cH}(\wt{A}-\lambda)^{-1}\cL:\lambda\in\rho(\wt{A})\}.
\end{equation}
\end{definition}
In this section we show that every $\cL$-minimal selfadjoint linear
relation $A$ is unitarily equivalent to its functional model
$A(\varphi,\psi)$, constructed in Theorem~\ref{thm:1.2}. The
operator $\cF:\cH\to\cH(\varphi,\psi)$ given by
\begin{equation}\label{eq:2.2}
h\mapsto(\cF h)(\lambda)=\gamma(\bar{\lambda})^*
h=P_{\cL}(\wt{A}-\lambda)^{-1}h \quad (h\in\cH)
\end{equation}
is called the {\it
generalized Fourier transform} associated with $\wt{A}$ and the
scale $\cL$.
\begin{theorem}\label{thm:2.1}
let $\wt{A}$ be a selfadjoint linear relation in $\cH\oplus\cL$ and let
$\{\varphi,\psi\}$ be the corresponding Nevanlinna pair given by~\eqref{eq:1.1}. Then:
\begin{enumerate}
\item[1)] The generalized Fourier transform $\cF$ maps isometrically
the subspace $\cH_0$ onto $\cH(\varphi,\psi)$ and $\cF$ is
identically equal to $0$ on $\cH\ominus\cH_0$.

\item[2)] For every $\left\{\left[%
\begin{array}{c}
  f \\
  u \\
\end{array}%
\right],\left[%
\begin{array}{c}
  f' \\
  u' \\
\end{array}%
\right]\right\}\in\wt{A}$ the following identity holds
\begin{equation}\label{eq:2.3}
E(\lambda)\cF(f'-\lambda f)=\begin{bmatrix}\varphi(\lambda)& -\psi(\lambda)\end{bmatrix}\left[%
\begin{array}{c}
  u \\
  u' \\
\end{array}%
\right].
\end{equation}
\end{enumerate}
\end{theorem}
\begin{proof}
1) For every vector $h=\gamma(\bar{\omega})u$
($\omega\in\rho(\wt{A})$, $u\in\cL$) it follows from
Proposition~\ref{pr:1.3} that
\[
(\cF h)(\lambda)= \gamma(\bar{\lambda})^*\gamma(\bar{\omega})u
={\mathsf N}^{\varphi\psi}_{\omega}(\lambda)u.
\]
Therefore, $\cF$ maps the linear space
$\mbox{span}\{\gamma(\omega)\cL:\omega\in\rho(\wt{A})\}$ dense in
$\cH_0$ onto the linear space $\mbox{span}\{{\mathsf
N}_{\omega}^{\varphi\psi}(\cdot)\cL:\omega\in\rho(\wt{A})\}$ which
is dense in $\cH(\varphi,\psi)$. Moreover, this mapping is
isometric, since
\begin{equation}\label{eq:2.31}
\begin{split}
 (\cF h,\cF h)_{\cH(\varphi,\psi)}
&=({\mathsf N}^{\varphi\psi}_{\omega}(\cdot)u,{\mathsf N}^{\varphi\psi}_{\omega}(\cdot)u)_{\cH(\varphi,\psi)}\\
&=({\mathsf N}^{\varphi\psi}_{\omega}(\omega)u,u)_{\cL}=(h,h)_{\cH}.
\end{split}
\end{equation}
This proves the first statement. It is clear from~\eqref{eq:2.2}
that $\cF h\equiv 0$ for $h\in\cH\ominus\cH_0$.

2) Let
\[
g=\gamma(\bar{\omega})v=P_{\cH}(\wt{A}-\bar{\omega})^{-1}v, \quad
v\in\cL.
\]
Then it follows from~\eqref{eq:1.1},~\eqref{eq:1.7} that
\[
\left[%
\begin{array}{c}
  g \\
  \psi(\bar{\omega})v \\
\end{array}%
\right]=(\wt{A}-\bar{\omega})^{-1}\left[%
\begin{array}{c}
  0 \\
  v \\
\end{array}%
\right],\quad\left[%
\begin{array}{c}
  \bar{\omega}g \\
  \varphi(\bar{\omega})v \\
\end{array}%
\right]=[I+\bar{\omega}(\wt{A}-\bar{\omega})^{-1}]\left[%
\begin{array}{c}
  0 \\
  v \\
\end{array}%
\right]
\]
and hence
\[
\left\{\left[%
\begin{array}{c}
  g \\
  \psi(\bar{\omega})v \\
\end{array}%
\right],\left[%
\begin{array}{c}
  \bar{\omega}g \\
  \varphi(\bar{\omega})v \\
\end{array}%
\right]\right\}\in\wt{A}.
\]
Since $\wt{A}=\wt{A}^*$ one obtains for all $\left\{\left[%
\begin{array}{c}
  f \\
  u \\
\end{array}%
\right],\left[%
\begin{array}{c}
  f' \\
  u' \\
\end{array}%
\right]\right\}\in\wt{A}$
\[
(f',g)_{\cH}-(f,\bar{\omega}g)_{\cH} +(u',{\psi}(\bar{\omega}))_{\cL}
-(u,{\varphi}(\bar\omega)v)_{\cL}=0,\quad v\in\cL.
\]
This implies
\begin{equation}\label{eq:2.5}
\gamma(\bar{\omega})^*(f'-\bar{\omega}f)=\varphi(\omega)u-\psi(\omega)u', \quad
\omega\in\dC \setminus\dR.
\end{equation}
This proves the identity~\eqref{eq:2.3}.
\end{proof}
\begin{remark}
In the case, when the linear relation $\wt{A}$ is $\cL$-minimal it
is unitary equivalent to the linear relation $A(\varphi,\psi)$ via
the formula
\begin{equation}\label{eq:2.6}
 \wt{A}(\varphi,\psi)=\left\{\left\{\left[%
\begin{array}{c}
  \cF f \\
  u \\
\end{array}%
\right],\left[%
\begin{array}{c}
  \cF f' \\
  u' \\
\end{array}%
\right]\right\}:\left\{\left[%
\begin{array}{c}
  f \\
  u \\
\end{array}%
\right],\left[%
\begin{array}{c}
  f' \\
  u' \\
\end{array}%
\right]\right\}\in\wt{A}\right\}.
\end{equation}
The operator $\cF\oplus I_{\cL}$ establishes this unitary
equivalence.
\end{remark}

\section{Abstract interpolation problem}

Let $\cX$ be a complex linear space, let $\cL$ be a Hilbert space, let $B_1$, $B_2$ be
linear operators in $\cX$, let $C_1$, $C_2$ be linear operators from $\cX$ to $\cH$, and
let $K$ be a nonegative sesquilinear form on $\cX$ which satisfies the following
identity
\begin{enumerate}
\item[(A1)]
$K(B_2h,B_1g)-K(B_1h,B_2g)=(C_1h,C_2g)_\cL-(C_2h,C_1g)_\cL$
\end{enumerate}
for every $h,g\in\cX$. Consider the following

\noindent {\bf Problem $AIP(B_1, B_2, C_1, C_2, K)$}. Let the data set $( B_1, B_2, C_1,
C_2, K)$ satisfies the assumption $(A1)$. Find a normalized Nevanlinna pair $
\{\varphi,\psi\} \in \wt{N}(\cL)$ such that for some linear mapping
$F:\cX\to\cH(\varphi,\psi)$ the following conditions hold:
\begin{enumerate}
\item[(C1)]
$(F{B_2h})(\lambda)-\lambda(F {B_1h})(\lambda)
=\left[\begin{array}{cc}\varphi(\lambda)&-\psi(\lambda)\end{array}\right]\left[%
\begin{array}{c}
  C_1 h \\
C_2 h \\
\end{array}%
\right]$;
\item[(C2)] $\|F h\|_{\cH(\varphi, \psi)}^2\leq K(h,h)$
\end{enumerate}
 for all $h\in \cX$.

Clearly $\ker K=\{h\in \cX:\,K(h,h)=0\}$ is a linear subspace of $\cX$.
Consider the factor space $\wh{\cX}=\cX/\ker K$ and denote by $\hat{h}$ the equivalence
class $h+\ker K$ in $\wh{\cX}$, $h\in \cX$. Let $\wh \cX$ be endowed with the scalar
product
\begin{equation}\label{scalarP}
(\wh h,\wh g)_{\cH}=K(h,g), \quad h, g\in  \cX
\end{equation}
and let $\cH$ be the completion of $\wh{\cX}$ with respect to the norm $\|h\|_{\cH}$.

In examples (see Section~\ref{Examples}) the linear space $\cX$ has an original inner
product. In order to avoid an ambiguity we denote by $B^+$ the adjoint to the operator
$B:\cH\to\cH$ and by $B^*$ the adjoint to  $B:\cX\to\cX$.
\begin{proposition}\label{pr:3.1}
Let the data set $(B_1,B_2,C_1,C_2,K)$ satisfies the assumption $(A1)$. Then the linear
relation
\begin{equation}\label{eq:3.5}
    \wh{A}=\left\{\left\{%
\left[\begin{array}{c}
  \widehat{B_1h} \\
  C_1 h \\
\end{array}%
\right],\left[%
\begin{array}{c}
  \widehat{B_2h} \\
  C_2 h \\
\end{array}%
\right]\right\}: h\in\cX\right\}
\end{equation}
is symmetric in ${\cH}\oplus\cL$.
\end{proposition}
\begin{proof}
The statement is implied by~\eqref{eq:3.5} since
\[
\begin{split}
&\left\langle\left[%
\begin{array}{c}
  \widehat{B_2h} \\
  C_2 h \\
\end{array}%
\right],\left[%
\begin{array}{c}
  \widehat{B_1h} \\
  C_1 h \\
\end{array}%
\right]\right\rangle_{\cH\oplus\cL}-\left\langle\left[%
\begin{array}{c}
  \widehat{B_1h} \\
  C_1 h \\
\end{array}%
\right],\left[%
\begin{array}{c}
  \widehat{B_2h} \\
  C_2 h \\
\end{array}%
\right]\right\rangle_{\cH\oplus\cL}\\
&= K(B_2h,B_1h)-K(B_1h,B_2h)-(C_1h,C_2h)_\cL+(C_2h,C_1h)_\cL=0.
\end{split}\]
\end{proof}
\begin{remark}
In general, the linear relation $\wh{A}$ need not be simple and its
deficiency indices not necessarily coincide.
\end{remark}
\begin{theorem}\label{thm:3.2}
Let the data set $(B_1,B_2,C_1,C_2,K)$ satisfies the assumption
$(A1)$. Then the Problem $AIP(B_1,B_2,C_1,C_2,K)$ is solvable and
the set of its normalized solutions is parametrized by the formula
\begin{equation}\label{eq:3.6}
\left[%
\begin{array}{c}
  \psi (\lambda) \\
  \varphi(\lambda) \\
\end{array}%
\right]=\left[%
\begin{array}{cc}
  I_\cL & 0 \\
  \lambda & I_\cL \\
\end{array}%
\right]\left[%
\begin{array}{c}
  P_{\cL}(\wt{A}-\lambda)^{-1}|_{\cL} \\
      I_{\cL} \\
\end{array}%
\right],
\end{equation}
where $\wt{A}$ ranges over the set of all selfadjoint extensions of $\wh{A}$ with the
exit in a Hilbert space $\wt{\cH}\oplus\cL\supset\cH\oplus\cL$. The corresponding linear
mapping $F:\cX\to\cH(\varphi,\psi)$ is given by
\begin{equation}\label{eq:3.61}
    (Fh)(\lambda)=P_\cL(\wt A-\lambda)^{-1}\wh h,\quad h\in \cX.
\end{equation}
\end{theorem}
\begin{proof}{\it Sufficiency.}
Let $\wt{A}$ be a selfadjoint extension of $\wh{A}$ and let $\{\varphi, \psi\}$ be the
normalized Nevanlinna pair corresponding to $\wt{A}$ and the scale $\cL$, and let
 $\cF:\wt{\cH}\rightarrow\cH(\varphi,\psi)$ be the corresponding
 generalized Fourier transform given by~\eqref{eq:2.2}. Then
the formula~\eqref{eq:3.6} is implied by~\eqref{eq:1.1} and in view of~\eqref{eq:2.2}
the linear mapping $F:\cX\to\cH(\varphi,\psi)$ given by~\eqref{eq:3.61} is connected to
$\cF$ via the formula
\begin{equation}\label{cFF}
    Fh=\cF\wh h,\quad (h\in \cX).
\end{equation}
Since $\cF$ satisfies the identity~\eqref{eq:2.3} and
 \[
\left\{%
\left[\begin{array}{c}
  \widehat{B_1h} \\
  C_1 h \\
\end{array}%
\right],\left[%
\begin{array}{c}
  \widehat{B_2h} \\
  C_2 h \\
\end{array}%
\right]\right\}\in {A}\subset\wt{A}
\]
one obtains from~\eqref{eq:2.3}
\begin{equation}\label{eq:3.7}
\begin{split}
(F {B_2h})(\lambda)-\lambda(F {B_1h})(\lambda)&= (\cF
\widehat{B_2h})(\lambda)-\lambda(\cF \widehat{B_1h})(\lambda)\\
&=\left[\begin{array}{cc}\varphi(\lambda)&-\psi(\lambda)\end{array}\right]\left[%
\begin{array}{c}
  C_1 h \\
C_2 h \\
\end{array}%
\right]
 \quad \forall h\in\cH.
\end{split}
\end{equation}
Next, it follows from~\eqref{scalarP} and Theorem~\ref{thm:2.1}
that
\[
\|Fh\|_{\cH(\varphi,\psi)}^2=\|\cF\widehat{h}\|_{\cH(\varphi,\psi)}^2\leq\|\wh
h\|_\cH^2=K(h,h).
\]
This proves $(C1)$, $ (C2)$ and, hence, $\{\varphi,\psi\}$ is a solution of the $AIP$.

{\it Necessity.}  Let a normalized Nevanlinna pair $\{\varphi,\psi\}$ be a solution of
the $AIP$ and let the mapping $F:\cX\rightarrow\cH(\varphi,\psi)$ satisfies (C1), (C2).
We will construct a selfadjoint exit space extension $\wt{A}$ of $\wh A$ such
that~\eqref{eq:3.6} and~\eqref{eq:3.61} hold.

 {\it Step 1.
Isometric embedding of $\cH$ into a Hilbert space.} It follows
from~$(C2)$ that
\begin{equation}\label{eq:3.15}
(Fh)(\lambda)\equiv 0 \mbox{ for all } h\in\ker K.
\end{equation}
Thus $F$ induces the mapping $\wh{F}:\wh \cX\rightarrow\cH(\varphi,\psi)$, which is well
defined by the equality
\begin{equation}\label{eq:3.16}
\hat{h}\rightarrow(\wh{F}\wh h)(\lambda)=(Fh)(\lambda),\quad h\in\cH
\end{equation}
and is contractive due to $(C2)$
\[
\|(\wh{F}\wh
h)(\lambda)\|_{\cH(\varphi,\psi)}^2=\|(Fh)(\lambda)\|_{\cH(\varphi,\psi)}^2\leq
K(h,h)=\|\wh h\|^2_\cH.
\]
We will keep the same notation for the continuous extension of $\wh F$ to $\cH$.

Let $ D= D^*(\ge 0)$ be the defect operator of the contraction
$\wh{F}$ defined by
\begin{equation}\label{eq:3.17}
 D^2=I-\wh{F^+}\wh{F}:\cH\rightarrow\cH
\end{equation}
and let $ \cD=\overline{\mbox{ran}}\, D$ be the defect subspace of
$\wh F$ in $\cH$. Consider the column extension $\wt{F}$ of the
operator $\wh{F}$ to the isometric mapping from $\cH$ to
$\cD\oplus\cH(\varphi,\psi)$ by the formula
\begin{equation}\label{eq:3.18}
\wt{F}{h}=\left[%
\begin{array}{c}
  D{h} \\
  \wh{F}{h} \\
\end{array}%
\right], \quad {h}\in{\cH}.
\end{equation}

{\it Step 2.} {\it Construction of a selfadjoint linear relation
$\wt{A}$.} Let $A_{ \cD}$ be a linear relation in $ \cD$ defined by
\[
A_{ \cD}=\left\{\left\{ D\wh{B_1 h}, D\wh{B_2h}\right\}:h\in \cX\right\}
\]
and let us show that $A_{ \cD}$ is symmetric in $ \cD$. Indeed, it
follows from~\eqref{eq:3.16},~\eqref{eq:3.17} that
\begin{equation}\label{eq:3.20}
\begin{split}
( D\widehat{B_2h}, D\widehat{B_1h})_\cH&-( D\widehat{B_1h}, D\widehat{B_2h})_\cH\\
&=((I-\wh{F}^+\wh{F})\widehat{B_2h},\widehat{B_1h})_\cH-((I-\wh{F}^+\wh{F})\widehat{B_1h},\widehat{B_2h})_\cH\\
&=K(B_2 h,B_1 h)-K(B_1 h,B_2h)\\
&-{(\wh{F}\widehat{B_2h},\wh{F}\widehat{B_1h})}_{\cH(\varphi,\psi)}
+(\wh{F}\widehat{B_1h},\wh{F}\widehat{B_2h})_{\cH(\varphi,\psi)}.
\\
\end{split}
\end{equation}
As follows from (C1)
\begin{equation}\label{eq:3.71}
\begin{split}
(\wh F\widehat{B_2h})(\lambda)-\lambda(\wh F \widehat{B_1h})(\lambda)
&= (F {B_2h})(\lambda)-\lambda(F {B_1h})(\lambda)\\
&= \left[\begin{array}{cc}\varphi(\lambda) & \psi(\lambda)\end{array}\right]\left[%
\begin{array}{c}
  C_1 h \\
C_2 h \\
\end{array}%
\right]
 \quad \forall h\in\cX.
\end{split}
\end{equation}
In view of~\eqref{eq:1.5} this implies
\[
\left\{\left[%
\begin{array}{c}
  \wh{F}\widehat{B_1h} \\
  C_1 h \\
\end{array}%
\right],\left[%
\begin{array}{c}
  \wh{F}\widehat{B_2h} \\
  C_2 h \\
\end{array}%
\right]\right\}\in A(\varphi,\psi).
\]
Therefore, the right hand part of~\eqref{eq:3.20} can be rewritten as
\[
K(B_2 h,B_1 h)-K(B_1 h,B_2 h)-{(C_1 h,C_2 h)_{\cL}+(C_2 h,C_1
h)_{\cL}}
\]
and hence it is vanishing due to~(A1).

Let $\wt A_\cD$ be a selfadjoint extension of $A_\cD$ in a Hilbert
space $\wt\cD\supset \cD$ and let $\wt A=\wt A_\cD\oplus
A(\varphi,\psi)$.

{\it Step 3. Linear relation $\wt{A}$ satisfies~\eqref{eq:3.6} and~\eqref{eq:3.61}}.
Under the identification of the vector $h\in\cH$ with $\wt Fh$ the symmetric linear
relation $\wh{A}$ in $\cH\oplus\cL$ can be identified with the symmetric linear relation
\[
\begin{split}
&A_1=(\wt{F}\oplus I_{\cL})\wh{A}(\wt{F}\oplus
I_{\cL})^{-1}\\
&=\left\{\left\{\left[%
\begin{array}{c}
  D\widehat{B_1 h} \\
  \widehat{F}\widehat{B_1 h} \\
  C_1 h \\
\end{array}%
\right],\left[%
\begin{array}{c}
  D\widehat{B_2 h} \\
  \widehat{F}\widehat{B_2 h} \\
  C_2 h \\
\end{array}%
\right]\right\}:h\in \cX\right\}\\
\end{split}
\]
in $\wt \cH:=\wt\cD\oplus\cH(\varphi,\psi)\oplus\cL$. Moreover $A_1$
is contained in the selfadjoint linear relation
\[
\wt{A} ={\wt A_{\cD}}\oplus A(\varphi,\psi),
\]
since $\{ D\widehat{B_1 h} ,D\widehat{B_2 h}  \}\in A_\cD\subset\wt
A_\cD$ and
\[
\left\{\left[%
\begin{array}{c}
  \wh{F}\widehat{B_1h} \\
  C_1 h \\
\end{array}%
\right],\left[%
\begin{array}{c}
  \wh{F}\widehat{B_2h} \\
  C_2 h \\
\end{array}%
\right]\right\}\in A(\varphi,\psi).
\]
 The
formula~\eqref{eq:3.6} is implied by the analogous formula for
$A(\varphi,\psi)$
\[
\left[%
\begin{array}{c}
  \psi(\lambda) \\
  \varphi(\lambda) \\
\end{array}%
\right]=\left[%
\begin{array}{cc}
  I_\cL & 0 \\
  \lambda I_\cL & I_{\cL} \\
\end{array}%
\right]\left[%
\begin{array}{c}
  P_{\cL}(A(\varphi,\psi)-\lambda)^{-1}|_{\cL} \\
      I_{\cL} \\
\end{array}%
\right]
\]
since
\[
P_{\cL}(A(\varphi,\psi)-\lambda)^{-1}|_{\cL}=P_{\cL}(\wt
A-\lambda)^{-1}|_{\cL}.
\]
Moreover, for every $h\in \cX$ one obtains
\[
\begin{split}
P_{\cL}(\wt A-\lambda)^{-1}\left[%
\begin{array}{c}
  \wt F\wh h \\
  0 \\
\end{array}%
\right]&=
P_{\cL}(\wt A-\lambda)^{-1}\left[%
\begin{array}{c}
    \wh Fh \\
  0 \\
\end{array}%
\right]\\
&=P_{\cL}( A(\varphi,\psi)-\lambda)^{-1}\left[%
\begin{array}{c}
 Fh \\
  0 \\
\end{array}%
\right]\\
&=\cF(\varphi,\psi)Fh=Fh.
 \end{split}
\]
This proves the formula~\eqref{eq:3.61}, since $ \wt F\wh h$ is
identified with $\wh h$.
\end{proof}
\begin{remark}\label{SchurNev}
    If $(B_1, B_2, C_1, C_2, K)$ is a data set for the AIP in
    Nevanlinna classes then the data set $(T_1, T_2, M_1, M_2, K)$
    given by
\[
\begin{split}
T_1&=B_1-iB_2,\quad T_2=B_1+iB_2;\\
M_1&=C_1-iC_2,\quad M_2=C_1+iC_2,
\end{split}
\]
is a data set for the AIP in the Schur class. Recall (see~\cite{KKhYu}), that a
contractive $[\cL]$-valued function $\omega(\zeta)$ is said to be a solution of the
problem $AIP(T_1, T_2, M_1, M_2, K)$, if there exists a map $\Phi$ from $\cX$ to the de
Branges Rovnyak space $\cH_\omega$, such that
\[
(\Phi T_1h)(t)-t (\Phi T_2h)(t)=\left[%
\begin{array}{cc}
  I & -\omega(t) \\
  -\omega^*(t) & I \\
\end{array}%
\right]
\left[%
\begin{array}{cc}
 M_2h\\
 M_1h \\
\end{array}%
\right],
\]
and $\|\Phi h\|^2_{\cH_\omega}\le K(h,h)$ for all $h\in\cX$. One can
check, that solutions of the problems $AIP(B_1, B_2, C_1, C_2, K)$
and $AIP(T_1, T_2, M_1, M_2, K)$ are related via some linear
fractional transformation and the result of Theorem~\ref{thm:3.2}
can be derived from the corresponding result in~\cite{KKhYu}.
However, we prefer to give a direct proof based on the de Branges
Rovnyak space $\cH(\varphi,\psi)$, especially since these spaces
will  be usefull in applications to some interpolation problems.
\end{remark}
In general, the mapping $F:\cX\to\cH(\varphi,\psi)$ in (C1), (C2) is not uniquely
defined by the solution $\{\varphi,\psi\}$ of the abstract interpolation problem
$AIP(B_1, B_2, C_1, C_2, K)$. We impose an additional assumption on the data set:
\begin{enumerate}
\item[(U)]
    $B_2-\lambda B_1$ is an isomorphism in $\cX$ for $\lambda$ in
    nonempty domains $\cO_\pm\subset\dC_\pm$,
\end{enumerate}
which ensures the uniqueness of $F$. Let us set
\[
G(\lambda)=\begin{bmatrix} C_1\\ C_2 \end{bmatrix}(B_2-\lambda B_1)^{-1}\quad
(\lambda\in\cO_\pm).
\]
\begin{proposition}\label{Funique}
    Let the data set $(B_1, B_2, C_1, C_2, K)$ satisfies the
    assumptions (A1), (U). Then the mapping $F:\cX\to\cH(\varphi,\psi)$ in (C1), (C2) is
    uniquely defined by the solution $\{\varphi,\psi\}$ of the
    abstract interpolation problem $AIP(B_1, B_2, C_1, C_2, K)$ by
    the formula
\begin{equation}\label{FG0}
    (Fh)(\lambda)=\begin{bmatrix} \varphi(\lambda) &
    -\psi(\lambda)\end{bmatrix}G(\lambda)h\quad (\lambda\in\cO_\pm).
\end{equation}
\end{proposition}
\begin{proof}
Applying (C1) to the vector
\[
h=h_\mu:=(B_2-\mu B_1)^{-1}g\quad (\mu\in\cO_\pm,\,g\in\cX),
\]
one obtains
\begin{equation}\label{eq:3.90}
\begin{split}
(F g)(\lambda) &=(F{B_2} h_{\mu})(\lambda)-\mu(F{B_1} h_{\mu})(\lambda)\\
&=(F{B_2} h_{\mu})(\lambda)-\lambda(F{B_1} h_{\mu})(\lambda)+(\lambda-\mu)(F{B_1}h_{\mu})(\lambda)\\
&=\begin{bmatrix} \varphi(\lambda) &
    -\psi(\lambda)\end{bmatrix}G(\mu)g+(\lambda-\mu)(F{B_1}h_{\mu})(\lambda).\\
\end{split}
\end{equation}
Setting in~\eqref{eq:3.90} $\lambda=\mu$, one obtains
\begin{equation}\label{eq:3.100}
(Fg)(\mu))=\begin{bmatrix} \varphi(\mu) &
    -\psi(\mu)\end{bmatrix}G(\mu)g\quad
(\mu\in\cO_{\pm}, g\in\cX).
\end{equation}
\end{proof}

\section{Description of solutions of abstract interpolation problem.}
In view of Theorem~\ref{thm:3.2} a description of the set of solutions of the AIP is
reduced to the description of $\cL$-resolvents of the linear relation $\widehat{A}$. The
later problem can be solved within the theory of $\cL$-resolvent matrix \cite{Kr1,KrSa}.
In this section we will treat both the nondegenerate $(\ker K\neq\{0\})$ and the
degenerate case $(\ker K\neq\{0\})$. In the case when the form $K(\cdot,\cdot)$ is
degenerate it will be more convenient to calculate the resolvent matrix of some
auxiliary linear relation $A_0$ which is a restriction of $\wh A$.
\subsection{Symmetric linear relation $A$ and boundary triplet for $A^+$.}
In this section we impose some additional assumptions on the data
set $(B_1$, $B_2$, $C_1$, $C_2$, $K)$.
\begin{enumerate}
\item[(A2)] $\dim\ker K<\infty$ and $\cX$ admits the representation
\begin{equation}\label{eq:$.30}
\cX=\cX_0\dotplus\ker K,
\end{equation}
such that $B_j\cX_0\subseteq\cX_0$ $( j=1,2)$.
\item[(A3)] $B_2=I_\cX$ and the operators $B_1|_{\cX_0}:\cX_0\subseteq\cH\to\cH$,
$C_1|_{\cX_0}$, $C_2|_{\cX_0}:\cX_0\subseteq\cH\to\cL$ are bounded.
\end{enumerate}
Due to the assumption (A2)  one can identify $\cX_0$ with $\wh
\cX=\cX/\ker K$ and consider the space $\cH$ as a completion of
$\cX_0$. The continuations of the operators $B_1|_{\cX_0}$,
$C_1|_{\cX_0}$, $C_2|_{\cX_0}$ will be denoted by $\wt B_1\in[\cH]$,
$\wt C_1$, $\wt C_2\in[\cH,\cL]$.

Define a linear relation
\begin{equation}\label{eq:4.32}
A_0=\left\{\left\{\left[%
\begin{array}{c}
  B_1 h \\
  C_1 h \\
\end{array}%
\right],\left[%
\begin{array}{c}
  B_2 h \\
  C_2 h \\
\end{array}%
\right]\right\}:h\in\cX_0\right\}
\end{equation}
in a Hilbert space $
{\cH}\oplus\cL$. Clearly, $A_0$ is a restriction of the symmetric
linear relation $\wh{A}$ which can be rewritten as
\begin{equation}\label{eq:4.33}
\wh{A}=\left\{\left\{\left[%
\begin{array}{c}
  B_1 h +\widehat{B_1u}\\
  C_1 h+C_1 u \\
\end{array}%
\right],\left[%
\begin{array}{c}
  B_2 h +\widehat{B_2u}\\
  C_2 h+C_2 u \\
\end{array}%
\right]\right\}:h\in\cX_0,u\in\ker K\right\}.
\end{equation}
In view of (A3) the closures of $A_0$ and $\wh A$ take the form
\begin{equation}\label{eq:4.34}
A:=\mbox{clos }A_0=\left\{\left\{\left[%
\begin{array}{c}
  \wt B_1 h \\
  \wt C_1 h \\
\end{array}%
\right],\left[%
\begin{array}{c}
h \\
  \wt C_2 h \\
\end{array}%
\right]\right\}:h\in\cH\right\},
\end{equation}
\begin{equation}\label{eq:4.35}
\mbox{clos }\wh{A}=\left\{\left\{\left[%
\begin{array}{c}
  \wt B_1 h +\widehat{B_1u}\\
  \wt C_1 h+C_1 u \\
\end{array}%
\right],\left[%
\begin{array}{c}
   h\\
  \wt C_2 h+C_2 u \\
\end{array}%
\right]\right\}:h\in\cH,u\in\ker K\right\}.
\end{equation}

Recall the definition of the boundary triplet for  nondensely
defined symmetric operator.
\begin{definition}{\rm \cite{GG},~\cite{MM2}}
\label{BT} A triplet $\Pi=\{{\cL},{\Gamma}_1,{\Gamma}_2\}$, where $\Gamma_i:\wh A^+ \to
\cL$, $i=0,1$, is said to be a \textit{boundary triplet} for $\wh A^+$, if  for all $\wh
f=\{f,f'\}$, $\wh g=\{g,g'\} \in \wh A^+$;
  \[
  (f', g)_{\cH\oplus\cL}-(f, g')_{\cH\oplus\cL}
  = (\Gamma_1\wh f, \Gamma_2\wh g)_{\cL}- (\Gamma_2\wh f, \Gamma_1\wh g)_{\cL}
  \]
and the mapping $\Gamma:=\begin{bmatrix}\Gamma_1\\
\Gamma_0\end{bmatrix}:\, \wh A^+ \to \begin{bmatrix}{\cL}\\
{\cL}\end{bmatrix}$ is surjective.
\end{definition}
A point $\lambda\in\dC$ is said to be a {\it regular type point} for a closed symmetric
linear relation $A$ if $\ran(A-\lambda)$ is closed in $\cH\oplus\cL$. It is well known
that the set $\wh\rho(A)$ of regular type points for symmetric linear relation $A$
contains $\dC_+\cup\dC_-$ and the defect subspaces
\[
\sN_\lambda(A):=(\cH\oplus\cL)\ominus\ran(A-\bar\lambda)
\]
have the same dimensions $n_+(A)$ and $n_-(A)$ for $\lambda\in\dC_+$ and
$\lambda\in\dC_-$, respectively. The numbers $n_+(A)$ and $n_-(A)$ are called the {\it
defect numbers} of the symmetric linear relation $A$. In the following proposition we
show, that the symmetric linear relation $A$ in~\eqref{eq:4.34} has equal defect numbers
$n_+(A)=n_-(A)=\dim \cL$ and, moreover, $0\in\wh\rho(A)$.
\begin{proposition}\label{pr:4.1} Let the data
set $(B_1$, $B_2$, $C_1$, $C_2$, $K)$ satisfy the assumptions
(A1)-(A3). Then:
\begin{enumerate}
\item[(1)]
the adjoint linear relation ${A}^+$ takes the form
\begin{equation}\label{eg:4.2}
{A}^+=\left\{\wh g=\left\{\left[%
\begin{array}{c}
  g \\
  v \\
\end{array}%
\right],\left[%
\begin{array}{c}
  g' \\
  v' \\
\end{array}%
\right]\right\}:
\begin{array}{c}
 v,v'\in\cL,
g'\in\cH; \\
 g=\wt B_1^+ g'+\wt C_1^+ v'-\wt C_2^+ v \\
\end{array}
\right\}
\end{equation}
\item[(2)]
the set $\wh\rho(A)$ of regular type points for symmetric linear relation $A$ contains
the resolvent set of the linear relation $\wt B_1^{-1}$
\[
\rho(\wt B_1^{-1})=\{\lambda\in(\dC\setminus \{0\}):1/\lambda\in\rho(\wt
B_1)\}\cup\{0\}\mbox{ if }\wt B_1\in[\cH]
\]
and the defect subspace $\sN_{\lambda}({A})$ for $\lambda\in\rho(\wt B_1^{-1})$ consists
of the vectors
\begin{equation}\label{eq:4.1}
\left[%
\begin{array}{c}
  -F(\overline{\lambda})^+ u \\
  u \\
\end{array}%
\right],\quad u\in\cL,
\end{equation}
where $F(\lambda)=(\wt C_2-\lambda \wt C_1)(I_\cH-\lambda \wt
B_1)^{-1}$;
\item[(3)]
a boundary triplet $\Pi=\{\cL,\Gamma_1,\Gamma_2\}$ for $\widehat{A}^+$ can be defined by
\begin{equation}\label{eq:4.3}
\Gamma_1\widehat{g}=-v+\wt C_1{g}',\quad \Gamma_2\widehat{g}=v'-\wt C_2{g}'.
\end{equation}
\end{enumerate}
\end{proposition}
\begin{proof}
1) Let
\[
\widehat{g}=\left\{\left[%
\begin{array}{c}
  g \\
  v \\
\end{array}%
\right],\left[%
\begin{array}{c}
  g' \\
  v' \\
\end{array}%
\right]\right\}\in{A}^+\quad (g,g'\in\cH; v,v'\in\cL).
\]
Then it follows from~\eqref{eq:4.34} that
\[
(g',\wt B_1 h)_{\cH}-(g, h)_{\cH}+(v',\wt C_1 h)_{\cL}-(v,\wt C_2 h)_{\cL}=0 
\]
for all $h\in\cH$ and, therefore,
\[
\wt B_1^+ g'-g+\wt C_1^+ v'-\wt C_2^+ v=0.
\]
This gives the equality
\begin{equation}\label{eq:4.4}
g=\wt B_1^+ g'+\wt C_1^+ v'-\wt C_2^+ v.
\end{equation}

2) It follows from~\eqref{eq:4.34} that
\begin{equation}\label{eq:4.36}
A-\lambda=\left\{\left\{\left[%
\begin{array}{c}
  \wt B_1 h \\
  \wt C_1 h \\
\end{array}%
\right],\left[%
\begin{array}{c}
(I_\cH-\lambda \wt
B_1)h \\
 (\wt C_2-\lambda \wt C_1) h \\
\end{array}%
\right]\right\}:h\in\cH\right\},
\end{equation}
 and, hence
\[
\ran(A-\lambda)=\left\{\begin{bmatrix} h\\
F(\lambda) h\end{bmatrix}:\,h\in\cH\right\}
\]
where $F(\lambda)=(\wt C_2-\lambda \wt C_1)(I_\cH-\lambda \wt B_1)^{-1}$. Therefore,
$\ran(A-\lambda)$ is closed for all $\lambda\in\rho(\wt B_1^{-1})$.

If $\lambda\in\rho(\wt B_1^{-1})$ and
$\widehat{g}\in\widehat{\sN}_{\lambda}(\widehat{A})$ then $g'=\lambda g$, $v'=\lambda
v$. Substituting these equalities in~\eqref{eq:4.4} one obtains
\[
(I_\cH-\lambda \wt B_1^+)g=-(\wt C_2^+-\lambda \wt C_1^+)v.
\]
This proves the second statement since
\[
g=-F(\bar{\lambda})^+ v.
\]
3) For two vectors
\[
\wh f=\left\{\left[%
\begin{array}{c}
  f \\
  v \\
\end{array}%
\right],\left[%
\begin{array}{c}
  f' \\
  v' \\
\end{array}%
\right]\right\},\quad
\widehat{g}=\left\{\left[%
\begin{array}{c}
  g \\
  v \\
\end{array}%
\right],\left[%
\begin{array}{c}
  g' \\
  v' \\
\end{array}%
\right]\right\}\in A^+
\]
one obtains
\begin{equation}\label{eq:4.5}
\begin{split}
( f',g)_\cH&-( f,g')_\cH+(u',v)_\cL-(u,v')_\cL=(u',v)_\cL-(u,v')_\cL
\\
&+(f',\wt B_1^+g'+\wt
 C_1^+v'-\wt C_2^+v)_\cH-(\wt B_1^+f'+\wt C_1^+u-\wt C_2^+u',g')_\cH.
\end{split}
\end{equation}
Then the right hand part of~\eqref{eq:4.5} takes the form
\[
\begin{split}
(\wt B_1f',g')_\cH&-(f',\wt B_1g')_\cH+(\wt C_1f',v')_\cL-(\wt C_2f',v)_\cL\\
&-(u',\wt C_1g')_\cL+(u,\wt C_2g')_\cL+(u',v)_\cL-(u,v')_\cL\\
&=(\wt C_2f',\wt C_1g')_\cL-(\wt C_1f',\wt C_2g')_\cL+(\wt C_1f',v')_\cL-(\wt C_2f',v)_\cL\\
&-(u',\wt C_1g')_\cL+(u,\wt C_2g')_\cL+(u',v)_\cL-(u,v')_\cL\\
&=(\wt C_2f'-u',\wt C_1g'-v)_\cL-(\wt C_1f'-u,\wt C_2g'-v')_\cL.
\end{split}
\]
Clearly, the mapping $\Gamma: A^+\to\cL\oplus\cL$ is surjective and hence the triplet
$\{\cL,\Gamma_1,\Gamma_2\}$  is a boundary triplet for $A^+$.
\end{proof}

\subsection{$\cL$-resolvent matrix} Recall some facts from M.G.~Krein~s
representation theory (\cite{Kr49}, \cite{DM95}) for symmetric linear relation $A$ in
$\cH\oplus\cL$. Let us say that $\lambda\in\rho(A,\cL)$ if $\lambda$ is a regular type
point for $ A$ and
\begin{equation}\label{eq:4.6}
\cH\oplus\cL=\ran( A-\lambda)\dotplus\cL.
\end{equation}
For every $\lambda\in\rho( A,\cL)$ the operator valued function
$\cP(\lambda):\cH\rightarrow\cL$ is defined as a skew projection
onto $\cL$ in the decomposition~\eqref{eq:4.6} and
$\cQ(\lambda):\cH\rightarrow\cL$ is given by
\begin{equation}\label{eq:4.7}
\cQ(\lambda)=P_{\cL}(A-\lambda)^{-1}(I-\cP(\lambda)),\quad
\lambda\in\rho( A,\cL).
\end{equation}
Let the operator-valued functions $\wh\cP(\lambda)^+$ and $\wh\cQ(\lambda)^+$ be given
by
\begin{equation}\label{eq:4.10}
\wh\cP(\lambda)^+u=\{\cP(\lambda)^+u,\bar{\lambda}\cP(\lambda)^+u\},\quad u\in\cL,
\end{equation}
\begin{equation}\label{eq:4.11}
\wh\cQ(\lambda)^+u=\{\cQ(\lambda)^+u,u+\bar{\lambda}\cQ(\lambda)^+u\}, \quad u\in\cL,
\end{equation}
where $\wh\cP(\lambda)^+$, $\wh\cQ(\lambda)^+:\cL\to\cH$ are adjoint operators to
$\cP(\lambda),\,\cQ(\lambda):\cH\rightarrow\cL$.   The role of the ovf's $\cP(\lambda)$,
$\cQ(\lambda)$ is clear from the following theorem.
\begin{theorem}\label{thm:4.2}{\rm \cite{Kr1,KrSa,DM95}}.
Let $\Pi=\{\cL,\Gamma_1,\Gamma_2\}$ be a boundary triplet for $A^+ $. The set of
$\cL$-resolvents of $ A$ is parametrized by the formula
\begin{equation}\label{eq:4.8}
P_{\cL}(\wt{A}-\lambda)^{-1}|_\cL
=(w_{11}(\lambda)q(\lambda)+w_{12}(\lambda)p(\lambda))
(w_{21}(\lambda)q(\lambda)+w_{22}(\lambda)p(\lambda))^{-1}
\end{equation}
where $(p,q)$ ranges over the set $\wt{\cN}(\cL)$ of Nevanlinna
pairs and the block matrix $W(\lambda)=(w_{ij}(\lambda))_{i,i=1}^2$
is given by
\begin{equation}\label{eq:4.9}
W_{\Pi\cL}(\lambda)=\left[%
\begin{array}{cc}
  -\Gamma_2\wh\cQ(\lambda)^* & \Gamma_2\wh\cP(\lambda)^*  \\
  -\Gamma_1\wh\cQ(\lambda)^*  & \Gamma_1\wh\cQ(\lambda)^*  \\
\end{array}%
\right]^*\quad \lambda\in\rho(\wh A,\cL).
\end{equation}
\end{theorem}
 In the following Theorem we calculate all the objects of M.G.Krein
representation theory $\cP(\lambda)$, $\cQ(\lambda)$ and the $\cL$-resolvent matrix
$W_{\Pi\cL}(\lambda)$ for the linear relation $ A$ in~\eqref{eq:4.34}.

\begin{theorem}\label{thm:4.3}
Let $B_1,B_2,C_1,C_2,K$ satisfy the assumptions $(A1)-(A3)$. Then:
\begin{enumerate}
\item
$ \rho( A,\cL)=\rho(B_1^{-1})$ and for $\lambda\in\rho( A,\cL),$ the ovf's
$\cP(\lambda)$, $\cQ(\lambda)$ are given by
\begin{equation}\label{eq:4.20}
\cP(\lambda)\left[%
\begin{array}{c}
  f \\
  u \\
\end{array}%
\right]=u-F(\lambda)f, \quad f\in\cH, \quad u\in\cL,
\end{equation}
\begin{equation}\label{eq:4.21}
\cQ(\lambda)\left[%
\begin{array}{c}
  f \\
  u \\
\end{array}%
\right]=\wt C_1(I_\cH-\lambda \wt B_1)^{-1}f, \quad f\in\cH;
\end{equation}
\item The adjoint operators to $\cP(\lambda),\cQ(\lambda):\left[%
\begin{array}{c}
  \cH \\
  \cL \\
\end{array}%
\right]\rightarrow\cL$ take the form
\begin{equation}\label{eq:4.22}
\cP(\lambda)^+u=\left[%
\begin{array}{c}
  -F(\lambda)^+u \\
  u \\
\end{array}%
\right] \quad u\in\cL,\,\lambda\in\rho( A,\cL),
\end{equation}
\begin{equation}\label{eq:4.23}
\cQ(\lambda)^+u=\left[%
\begin{array}{c}
 (I_\cH-\bar{\lambda}\wt B_1^+)^{-1}\wt C_1^+u \\
  0 \\
\end{array}%
\right] \quad u\in\cL,\,\lambda\in\rho(A,\cL);
\end{equation}
\item
The $\cL$-resolvent matrix $W_{\Pi\cL}(\lambda)$ corresponding to
the boundary triplet $\Pi$ takes the form
\begin{equation}\label{eq:4.24}
W_{\Pi\cL}(\lambda)=\left[%
\begin{array}{cc}
  -I_\cL & 0 \\
  \lambda & -I_\cL \\
\end{array}%
\right]\left( I_{\cL\oplus\cL}+i\lambda\left[%
\begin{array}{c}
  \wt C_1 \\
  \wt C_2 \\
\end{array}%
\right](I_\cH-\lambda \wt B_1)^{-1}\left[%
\begin{array}{c}
  \wt C_1 \\
  \wt C_2 \\
\end{array}%
\right]^+J \right).
\end{equation}
\end{enumerate}
\end{theorem}
\begin{proof}
1) Assume that $\lambda\in\rho(A,\cL)$ and the
decomposition~\eqref{eq:4.6} holds. Then for $f\in\cH$, $u\in\cL$
there are unique $h\in\cH$ and $v\in\cL$ such that
\begin{equation}\label{rAL}
  (I_\cH-\lambda \wt B_1)h=f,\quad (\wt C_2-\lambda \wt C_1)h+v=u.
\end{equation}
This implies, in particular, that $\lambda\in\rho(B_1^{-1})$.
Conversely, if $\lambda\in\rho(B_1^{-1})$, then the
equations~\eqref{rAL} have unique solutions $h\in\cH$ and $v\in\cL$
and, hence, $\lambda\in\rho(A,\cL)$. In view of~\eqref{rAL} these
solutions take the form
\begin{equation}\label{eq:4.25}
h=(I_\cH-\lambda \wt B_1)^{-1}f, \quad v=\cP(\lambda)\left[%
\begin{array}{c}
  f \\
  u \\
\end{array}%
\right]= u-F(\lambda)f.
\end{equation}
It follows from~\eqref{eq:4.7}, \eqref{eq:4.25} and~\eqref{eq:4.36}
that
\[
\begin{split}
\cQ(\lambda)\left[%
\begin{array}{c}
  f \\
  u \\
\end{array}%
\right] &=P_{\cL}(A-\lambda)^{-1}\left[%
\begin{array}{c}
  f \\
  F(\lambda)f \\
\end{array}%
\right]=P_{\cL}\left[%
\begin{array}{c}
  \wt B_1 h \\
  \wt C_1 h \\
\end{array}%
\right]\\
&=\wt C_1(I_\cH-\lambda \wt B_1)^{-1}f.
\end{split}
\]

 2) The formulas~\eqref{eq:4.22},~\eqref{eq:4.23} are
implied by
\[
\begin{split}
\left\langle\cP(\lambda)^+v,\left[%
\begin{array}{c}
  f \\
  u \\
\end{array}%
\right]\right\rangle_{\cH\oplus\cL}&=(v,u-F(\lambda)f)_\cL\\
&=\left\langle\left[%
\begin{array}{c}
  -F(\lambda)^+v \\
  v \\
\end{array}%
\right],\left[%
\begin{array}{c}
  f \\
  u \\
\end{array}%
\right]\right\rangle_{\cH\oplus\cL}, \end{split}
\]
\[
\begin{split}
\left\langle\cQ(\lambda)^+v,\left[%
\begin{array}{c}
  f \\
  u \\
\end{array}%
\right]\right\rangle_{\cH\oplus\cL}&=( v,\wt C_1(I_\cH-\lambda
\wt B_1)^{-1}f)_\cL\\
&=\left\langle\left[%
\begin{array}{c}
  (I_\cH-\lambda\wt B_1^+)^{-1}\wt C_1^+v \\
  0 \\
\end{array}%
\right],\left[%
\begin{array}{c}
  f \\
  u \\
\end{array}%
\right]\right\rangle_{\cH\oplus\cL}.
\end{split}
\]
Now one obtains from~\eqref{eq:4.10},~\eqref{eq:4.22}
and~\eqref{eq:4.3} that
\begin{equation}\label{eq:4.26}
\Gamma_2\wh{\cP}(\lambda)^+ v=\bar\lambda v+\bar{\lambda}\wt C_2 F(\lambda)^+ v,
\end{equation}
\begin{equation}\label{eq:4.27}
\Gamma_1\wh{\cP}(\lambda)^+ v=-v-\bar{\lambda}\wt C_1 F(\lambda)^+ v.
\end{equation}
Similarly~\eqref{eq:4.11},~\eqref{eq:4.23} and~\eqref{eq:4.3} imply
\begin{equation}\label{eq:4.28}
\Gamma_2\wh{\cQ}(\lambda)^+ v= v-\bar{\lambda}\wt C_2 (I_\cH-\lambda\wt B_1^+)^{-1}\wt
C_1^+ v,
\end{equation}
\begin{equation}\label{eq:4.29}
\Gamma_1\wh{\cQ}(\lambda)^+ v=\bar{\lambda}\wt C_1 (I_\cH-\lambda\wt B_1^+)^{-1}\wt
C_1^+ v.
\end{equation}
It follows from~\eqref{eq:4.26}-~\eqref{eq:4.29} and~\eqref{eq:4.9}
that
\[
W_{\Pi\cL}(\lambda)^*=\left[%
\begin{array}{cc}
  -I_\cL & \bar\lambda \\
  0 & -I_\cL \\
\end{array}%
\right]+\bar\lambda\left[%
\begin{array}{c}
  \wt C_2 \\
  -\wt C_1 \\
\end{array}%
\right](I_\cH-\bar\lambda\wt B_1^+)^{-1} \begin{bmatrix}\wt C_1^+ & \wt
C_2^+-\bar{\lambda}\wt C_1^+\end{bmatrix}
\]
and hence
\[
\begin{split}
W_{\Pi\cL}(\lambda)
&=\left[%
\begin{array}{cc}
  -I & 0 \\
  \lambda & -I \\
\end{array}%
\right]+\lambda\left[%
\begin{array}{c}
  \wt C_1 \\
  \wt C_2-\lambda \wt C_1 \\
\end{array}%
\right](I_\cH-\lambda \wt B_1)^{-1}\begin{bmatrix}\wt C_2^+ &
-\wt C_1^+\end{bmatrix}\\
&=\left[%
\begin{array}{cc}
  -I & 0 \\
  \lambda & -I \\
\end{array}%
\right]\left( I_{\cL\oplus\cL}-\lambda\left[%
\begin{array}{c}
  \wt C_1 \\
  \wt C_2 \\
\end{array}%
\right](I_\cH-\lambda \wt B_1)^{-1}\begin{bmatrix}\wt C_2^+ & -\wt
C_1^+\end{bmatrix}\right).
\end{split}
\]
\end{proof}
\begin{corollary}\label{cor:4.4}
Let the data set $(B_1,B_2,C_1,C_2,K)$ satisfy $(A1),(A2)$ and
\begin{enumerate}
\item[$(A3')$] the operators $B_1|_{\cX_0},B_2|_{\cX_0}:\cX_0\subset\cH\to\cH$,
$C_1|_{\cX_0}$, $C_2|_{\cX_0}:\cX_0\subset\cH\to\cL$ are bounded,
$(B_2-\mu B_1)\cX=\cX$  for some $\mu\in\dR$, and the continuations
$\wt B_1$, $\wt B_2$ of the operators $B_1|_{\cX_0}$, $B_2|_{\cX_0}$
satisfy the condition $0\in\rho(\wt B_2-\mu \wt B_1)$.
\end{enumerate}
Then one of the $\cL$-resolvent matrices can be found from
\begin{equation}\label{eq:W.26}
\begin{split}
&\left[%
\begin{array}{cc}
  -I & 0 \\
  \lambda & -I \\
\end{array}%
\right]^{-1}W^{\mu}(\lambda)
\\
&=I+i(\lambda-\mu)\left[%
\begin{array}{c}
  \wt C_1 \\
  \wt C_2 \\
\end{array}%
\right](\wt B_2-\lambda \wt B_1)^{-1}(\wt B_2^+-\mu \wt B_1^+)^{-1}\left[%
\begin{array}{c}
  \wt C_1 \\
  \wt C_2 \\
\end{array}%
\right]^+J .
\end{split}
\end{equation}
\end{corollary}
\begin{proof}
The data set
\[
(B_1(B_2-\mu B_1)^{-1}, I_\cX, C_1(B_2-\mu B_1)^{-1}, (C_2-\mu C_1)(B_2-\mu
B_1)^{-1},K)
\]
satisfies the assumptions (A1)-(A3). Consider the linear relation $ A-\mu$
\[
A-\mu=\left\{\left\{\left[%
\begin{array}{cc}
  \wt B_1(\wt B_2-\mu \wt B_1)^{-1} h \\
  \wt C_1(\wt B_2-\mu \wt B_1)^{-1} h \\
\end{array}%
\right],\left[%
\begin{array}{c}
  h \\
  (\wt C_2-\mu \wt C_1)(\wt B_2-\mu \wt B_1)^{-1}h \\
\end{array}%
\right]\right\}:h\in\cH\right\}.
\]
 Due to~\eqref{eq:4.25} its
$\cL$-resolvent matrix $W(\lambda)$ satisfies the equality
\[
\begin{split}
&\left[%
\begin{array}{cc}
  -I & 0 \\
  \lambda & -I \\
\end{array}%
\right]^{-1}W(\lambda)=\\
&=I_{\cL\oplus\cL}+i\lambda\left[%
\begin{array}{c}
  \wt C_1 \\
  \wt C_2-\mu \wt C_1 \\
\end{array}%
\right](\wt B_2-(\lambda+\mu)\wt B_1)^{-1}(\wt B_2^+-\mu\wt B_1^+)^{-1}\left[%
\begin{array}{c}
  \wt C_1 \\
  \wt C_2-\mu\wt C_1 \\
\end{array}%
\right]^+J .
\end{split}
\]
Then the matrix $W^{\mu}(\lambda)=W(\lambda-\mu)$ is the $\cL$-resolvent matrix of $A$
and, hence,
\[
\begin{split}
&\left[%
\begin{array}{cc}
  -I & 0 \\
  \lambda & -I \\
\end{array}%
\right]^{-1}W^{\mu}(\lambda)=\left[%
\begin{array}{cc}
  -I & 0 \\
  \lambda & -I \\
\end{array}%
\right]^{-1}W(\lambda-\mu)\\
&=\left(I_{\cL\oplus\cL}+i(\lambda-\mu)\left[%
\begin{array}{c}
  \wt C_1 \\
  \wt C_2\\
\end{array}%
\right](\wt B_2-\lambda \wt B_1)^{-1}(\wt B_2^+-\mu \wt B_1^+)^{-1}
\left[%
\begin{array}{c}
  \wt C_1 \\
  \wt C_2 \\
\end{array}%
\right]^+J\right)\left[%
\begin{array}{cc}
  I & 0 \\
  \mu & I \\
\end{array}%
\right].
\end{split}
\]
\end{proof}
\subsection{$\cL$-resolvents of $\wh A$.}
In the case when $\ker P$ is nontrivial we calculated the $\cL$-resolvent matrix of the
linear relation $A_0(\subset\wh A)$. A Description of $\cL$-resolvents of $A$ is given
in Theorem~\ref{thm:4.2}. In order to obtain a description of $\cL$-resolvents of $\wh
A$ we will use the same formula and specify the set of parameters $\{p,g\}\in\wt N(\cL)$
which correspond to $\cL$-resolvents of $\wh A$ via~\eqref{eq:4.8}.

Recall (see~\cite{Str}) that every generalized resolvent $P_\cH(\wt A-\lambda)^{-1}|\cH$
of $A$ corrresponding to an exit space selfadjoint extension $\wt A$ in a Hilbert space
$\wt \cH=\cH\oplus\cH_1$ can be represented as
\begin{equation}\label{GR}
P_\cH(\wt A-\lambda)^{-1}|\cH=(T(\lambda)-\lambda)^{-1},\quad \lambda\in\dC_+,
\end{equation}
where $T(\lambda)$ $(\lambda\in\dC_+)$ is the {\it Strauss family} of maximal
dissipative linear relations in $\cH$ defined by
\begin{equation}\label{StrE}
    T(\lambda)=\{\{Pf,Pf'\}:\{f,f'\}\in\wt A,\,f'-\lambda f\in\cH\}.
\end{equation}

\begin{proposition}\label{StrFam}
{\rm (\cite{Bruk76}, \cite{DM95})} Let $\wt A$ be an exit space selfadjoint extension of
$A$, let $T(\lambda)$ be the Strauss family of maximal dissipative linear relations
defined by~\eqref{StrE}, let $\{\cL,\Gamma_1,\Gamma_2\}$ be a boundary triplet for
$A^+$. Then the pair $\{p,q\}\in\wt N(\cL)$ is the Nevanlinna pair corresponding to $\wt
A$ via ~\eqref{eq:4.8} if and only if the pair $\{p,q\}$ is related to $T(\lambda)$ via
the formula
\[
\Gamma T(\lambda)=\ran\begin{bmatrix}q(\lambda)\\
p(\lambda)\end{bmatrix} .
\]
\end{proposition}
We will need the following simple statement
\begin{lemma}\label{GT}
Let under the assumptions of Proposition~\ref{StrFam} $\wh A$ be a symmetric extension
of $A$. Then:
\begin{enumerate}
\item[(i)]
  $\wh A\subset \wt A$ if and only if $\wh A\subset T(\lambda)$ for some
  $\lambda\in\dC_+$;
\item[(ii)]
  $\wh A\subset \wt A$ if and only if $\Gamma\wh A\subset \ran\begin{bmatrix}q(\lambda)\\
p(\lambda)\end{bmatrix}$ for some
  $\lambda\in\dC_+$.
\end{enumerate}
\end{lemma}
\begin{proof}
1) The implication $\Rightarrow$ is immediate from~\eqref{StrE}. Conversely, assume that
$\wh A\subset T(\lambda)$. In view of~\eqref{StrE} for every $\{g,g'\}\in\wh A$ there
are $\{f,f'\}\in\wt A$ and $g_1\in\cH_1$ such that
\begin{equation}\label{ffgg}
    f=g+g_1,\quad f'=g'+\lambda g_1.
\end{equation}
Hence
\[
(f',f)_{\wt\cH}=(g',g)_\cH+\lambda(g_1,g_1)_{\cH_1}.
\]
Since $\wh A$ and $\wt A$ are symmetric this implies $g_1=0$. Therefore, $\{g,g'\}\in\wt
A$ and hence $\wh A\subset\wt A$.

2) The statement (ii) is implied by (i) since the inclusion $\wh A\subset T(\lambda)$ is
equivalent to $\Gamma\wh A\subset \Gamma T(\lambda)=\ran\begin{bmatrix}q(\lambda)\\
p(\lambda)\end{bmatrix}$.
\end{proof}
It follows from Lemma~\ref{GT} that all Nevanlinna pairs corresponding to generalized
resolvents of $\wh A$ have a common constant part $\Gamma\wh A$.
\begin{lemma}\label{lem:V}
Let $\wh A$ be the symmetric linear relation~\eqref{eq:4.33}, and let
$\{\cL,\Gamma_1,\Gamma_2\}$ be a boundary triplet for $A^+$. Then
\begin{enumerate}
\item[(i)]
  $\Gamma\wh
A$ is a neutral subspace in $(\cL^2,J_\cL)$ of dimension $\nu:=\dim C\ker K$;
\item[(ii)]
  There is a subspace $\cL_0\!\subset\!\cL\!$ and a $\!J_\cL\!$-unitary operator $\!V\in[\cL^2]$
  such that
\[
V(\{0\}\times \cL_0)\!=\!\Gamma\wh A.
\]
\end{enumerate}
\end{lemma}
\begin{proof}
1) It follows from~\eqref{eq:4.35} and~\eqref{eq:4.3} that
\[
\Gamma\wh A=\left\{\left[\begin{array}{c}
                            \Gamma_1\wh g \\
                            \Gamma_2\wh g
                          \end{array}\right]
:\,\wh g\in\wh A\right\}= \left\{\left[\begin{array}{r}
                            C_1u \\
                            -C_2u
                          \end{array}\right]
:\,u\in\ker K\right\}.
\]
Clearly, the subspace $\Gamma\wh A$ is finite-dimensional and
\[
\dim \Gamma\wh A= \dim C\ker K.
\]
The subspace $\Gamma\wh A$ is neutral since for every $u\in\ker K$ one has
\[
\begin{split}
(\Gamma_1\wh g,\Gamma_2\wh g)-(\Gamma_1\wh g,\Gamma_2\wh g)&=(C_1u,C_2u)_\cL-(C_2u,C_1u)_\cL\\
&=K(u,B_1u)-K(B_1u,u)=0.
\end{split}
\]
2) Let us decompose $\cL$ into the orthogonal sum
\[
\cL=\cL_0\oplus\cL_1,
\]
where $\cL_0$ is a subspace of $\cL$ such that $\dim \cL_0=\nu$. The subspace
$\{0\}\times \cL_0$ is $J_\cL$-neutral and hence there exists a $\!J_\cL\!$-unitary
operator $\!V\in[\cL^2]$
  such that
$V(\{0\}\times \cL_0)\!=\!\Gamma\wh A$.
\end{proof}
Let $V$ be the $J_\cL$-unitary operator, constructed in Lemma~\ref{lem:V}. Then
\begin{equation}\label{eq:What}
\wh W(\lambda)=(\wh w_{ij}(\lambda))_{i,j=1}^2:=W_{\Pi\cL}(\lambda)V.
\end{equation}
is also the $\cL$-resolvent matrix of $A_0$ with the advantage that the $\cL$-resolvents
of $\wh A$ can be easily described in its terms.
\begin{proposition}\label{prop:4.21}
The set of all $\cL$-resolvents of $\wh A$ is parametrized by the formula
\begin{equation}\label{eq:4.81}
P_{\cL}(\wt{A}-\lambda)^{-1}|_\cL =(\wh w_{11}(\lambda) q(\lambda)+\wh w_{12}(\lambda)
p(\lambda)) (\wh
 w_{21}(\lambda) q(\lambda)+\wh w_{22}(\lambda) p(\lambda))^{-1}
\end{equation}
where $\{ p, q\}$ ranges over the set $\wt{\cN}(\cL)$ of Nevanlinna pairs of the form
\begin{equation}\label{pq}
     p(\lambda)=\left[%
\begin{array}{cc}
  I_{\cL_0} & 0 \\
  0 & p_1(\lambda) \\
\end{array}%
\right],\quad
 q(\lambda)=\left[%
\begin{array}{cc}
  0_{\cL_0} & 0 \\
  0 & q_1(\lambda) \\
\end{array}%
\right],\quad \{p_1,q_1\}\in\wt N(\cL_1).
\end{equation}
\end{proposition}
\begin{proof}
Let $\{\wt p,\wt q\}$ be a Nevanlinna pair defined by
\[
\begin{bmatrix}
\wt q(\lambda)\\
\wt p(\lambda)
\end{bmatrix}=V\begin{bmatrix}
 q(\lambda)\\
 p(\lambda)
\end{bmatrix},\quad \{ p, q\}\in\wt{\cN}(\cL).
\]
It follows from Lemma~\ref{GT} that the formula
\[
P_{\cL}(\wt{A}-\lambda)^{-1}|_\cL =( w_{11}(\lambda)\wt q(\lambda)+ w_{12}(\lambda) \wt
p(\lambda)) ( w_{21}(\lambda)\wt q(\lambda)+w_{22}(\lambda)\wt p(\lambda))^{-1}
\]
establishes a 1-1 correspondence between the set of all $\cL$-resolvents of $\wh A$ and
the set of Nevanlinna families $\{\wt p,\wt q\}$ such that
\begin{equation}\label{eq:4.50}
    \Gamma\wh A\subset \ran\begin{bmatrix}\wt q(\lambda)\\
\wt p(\lambda)\end{bmatrix},\quad\lambda\in\dC_+.
\end{equation}
Since
\[
\Gamma\wh A=V\begin{bmatrix}0\\\cL_0\end{bmatrix},\quad\mbox{and }
\ran\begin{bmatrix}\wt q(\lambda)\\
\wt p(\lambda)\end{bmatrix}=V\ran\begin{bmatrix} q(\lambda)\\
 p(\lambda)\end{bmatrix},
\]
the inclusion~\eqref{eq:4.50} is equivalent to the inclusion
\begin{equation}\label{eq:4.51}
    \begin{bmatrix}0\\ \cL_0\end{bmatrix}\subset \ran\begin{bmatrix} q(\lambda)\\
 p(\lambda)\end{bmatrix},\quad\lambda\in\dC_+,
\end{equation}
which, in turn, means that the pair $\{p,q\}$ admits the representation~\eqref{pq}.
\end{proof}

\subsection{Description of AIP solutions.}
To describe solutions of the AIP it remains to combine Theorem~\ref{thm:3.2} and
Proposition~\ref{prop:4.21}. Let the ovf $\Theta(\lambda)$ be defined by
\begin{equation}\label{eq:Theta}
\Theta(\lambda)=\left[%
\begin{array}{cc}
  I & 0 \\
  {\lambda} & I \\
\end{array}%
\right]\wh W(\lambda)=\left(I_{\cL\oplus\cL}-\lambda\left[%
\begin{array}{c}
  \wt C_1 \\
  \wt C_2 \\
\end{array}%
\right](I_\cH-\lambda \wt B_1)^{-1}\begin{bmatrix}\wt C_2^+ & -\wt
C_1^+\end{bmatrix}\right)V
\end{equation}
\begin{theorem}\label{Lres2}
Let the AIP data set satisfy $(A1)-(A3)$.
 Then the formula
\begin{equation}\label{AIPdescr}
    \left[%
\begin{array}{c}
  \psi(\lambda) \\
  \phi(\lambda) \\
\end{array}%
\right]=\Theta(\lambda)
\left[%
\begin{array}{c}
  q(\lambda) \\
  p(\lambda) \\
\end{array}%
\right](\wh w_{21}(\lambda)q(\lambda)+\wh w_{22}(\lambda)p(\lambda))^{-1}
\end{equation}
 establishes the 1-1 correspondence between the set of all normalized solutions
$\{\phi,\psi\}$ of $AIP(B_1,B_2,C_1,C_2,K)$ and the set of all equivalence classes of
Nevanlinna pairs $\{p,q\}\in\wt N(\cL)$ of the form~\eqref{pq}. The corresponding
mapping $F:\cX\to\cH(\phi,\psi)$ in (C1), (C2) is uniquely defined by the solution
$\{\phi,\psi\}$:
\begin{equation}\label{FG}
    (Fh)(\lambda)=\begin{bmatrix} \varphi(\lambda) &
    -\psi(\lambda)\end{bmatrix}\wt G(\lambda)\cP_{\cX_0}h\quad (\lambda\in\cO_\pm),
\end{equation}
where
\[
\wt G(\lambda)=\begin{bmatrix} \wt C_1\\ \wt C_2 \end{bmatrix}(I_\cH-\lambda \wt
B_1|_{\cX_0})^{-1}\quad (\lambda\in\cO_\pm),
\]
and $\cP_{\cX_0}$ is the skew projection onto $\cX_0$ in the
decomposition~\eqref{eq:$.30}.
\end{theorem}
\begin{proof}
Indeed the description~\eqref{AIPdescr} is implied by~\eqref{eq:3.6},~\eqref{eq:4.81}
and~\eqref{eq:Theta}.

Next, let $g\in (B_2-\mu B_1)\cX_0$ $(\mu\in\cO_\pm)$.  Applying (C1) to the vector
\[
h=h_\mu:=(B_2-\mu B_1)^{-1}g,
\]
one obtains
\begin{equation}\label{eq:3.9}
\begin{split}
(F g)(\lambda) &=(F{B_2} h_{\mu})(\lambda)-\mu(F{B_1} h_{\mu})(\lambda)\\
&=\begin{bmatrix} \varphi(\lambda) &
    -\psi(\lambda)\end{bmatrix}G(\mu)g+(\lambda-\mu)(F{B_1}h_{\mu})(\lambda).\\
\end{split}
\end{equation}
Setting in~\eqref{eq:3.9} $\lambda=\mu$, one obtains
\begin{equation}\label{eq:3.10}
(Fg)(\mu))=\begin{bmatrix} \varphi(\mu) &
    -\psi(\mu)\end{bmatrix}\wt G(\mu)g\quad
(\mu\in\cO_{\pm}, g\in\cX).
\end{equation}
The equality ~\eqref{eq:3.10} holds for every $g\in(B_2-\mu B_1)\cX_0$
$(\mu\in\cO_\pm)$.

Let $g\in\cX_0$ and let $g_n\in(B_2-\mu B_1)\cX_0$ and $g_n\to g$. Then taking the limit
in
\[
(Fg_n)(\mu)=\begin{bmatrix} \varphi(\mu) &
    -\psi(\mu)\end{bmatrix}\wt G(\mu)g_n
\]
one obtains~\eqref{FG} for $g\in\cX_0$. To prove~\eqref{FG}  it remains to notice that
$Fg\equiv 0$ for all $g\in\ker K$.
\end{proof}
\begin{theorem}\label{Lres1}
Let the AIP data set satisfy $(A1)$, $(A2)$, $(A3')$ and let
\[
\begin{split}
\Theta^\mu(\lambda)&=\left[%
\begin{array}{cc}
  I & 0 \\
  {\lambda} & I \\
\end{array}%
\right]\wh W^\mu(\lambda)\\
&=\left(I+i(\lambda-\mu)\left[%
\begin{array}{c}
  \wt C_1 \\
  \wt C_2 \\
\end{array}%
\right](\wt B_2-\lambda \wt B_1)^{-1}(\wt B_2^+-\mu \wt B_1^+)^{-1}\left[%
\begin{array}{c}
  \wt C_1 \\
  \wt C_2 \\
\end{array}%
\right]^+J \right)V.
\end{split}
\]
 Then the formula
\begin{equation}\label{AIPdescr1}
    \left[%
\begin{array}{c}
  \psi(\lambda) \\
  \phi(\lambda) \\
\end{array}%
\right]=\Theta^\mu(\lambda)
\left[%
\begin{array}{c}
  q(\lambda) \\
  p(\lambda) \\
\end{array}%
\right](\wh w_{21}^\mu(\lambda)q(\lambda)+\wh w_{22}^\mu(\lambda)p(\lambda))^{-1}
\end{equation}
establishes the 1-1 correspondence between the set of all normalized solutions
$\{\phi,\psi\}$ of $AIP(B_1,B_2,C_1,C_2,K)$ and the set of all equivalence classes of
Nevanlinna pairs $\{p,q\}\in\wt N(\cL)$ of the form~\eqref{pq}. The corresponding
mapping $F:\cX\to\cH(\phi,\psi)$ in (C1), (C2) is uniquely defined by the
formula~\eqref{FG}.
\end{theorem}
\begin{corollary}\label{mdescr}
Let the AIP data set satisfies $(A1)-(A3)$ and let $\ran C_2=\cL$.
 Then the formula
\begin{equation}\label{Mdescr}
    m(\lambda) =(\theta_{11}(\lambda)q(\lambda)+ \theta_{12}(\lambda)p(\lambda))
    (\theta_{21}(\lambda)q(\lambda)+ \theta_{22}(\lambda)p(\lambda))^{-1}
\end{equation}
 establishes the 1-1 correspondence between the set of all solutions
$m(\lambda)$ of \linebreak
$AIP(B_1,B_2,C_1,C_2,K)$ and the set of all equivalence
classes of Nevanlinna pairs $\{p,q\}\in\wt N(\cL)$ of the form~\eqref{pq}.
\end{corollary}
In the case when the AIP data set satisfy $(A1)$, $(A2)$, and $(A3')$ similar formula
can be written in terms of the mvf $\Theta^\mu(\lambda)$.

\section{Examples}\label{Examples}
\subsection{Tangential interpolation problem.} Let $
\lambda_j\in\dC_+$, $\xi_j\in\dC^d$, $\eta_j\in\dC^d$ $(1\le j\le n)$. Consider the
following problem. Find $m\in N^{d\times d}:=N(\dC^d)$ such that
\begin{equation}\label{Tan1}
m(\lambda_j)\eta_j=\xi_j\quad (1\le j\le n).
\end{equation}
The problem~\eqref{Tan1} is called {\it tangential} (or one-sided) interpolation problem
and was considered first by I. Fedchina~\cite{Fed72} in the Schur class.
In~\cite{KKhYu}, \cite{Kh1} the inclusion of this (and more general bitangential)
problem into the scheme of AIP was demonstrated.

For the case of Nevanlinna class let us set $B_1=I_n$,
\[
   B_2=\mbox{diag }
    (\lambda_1\,\dots\,\lambda_n ),\quad
   C_1=\left[\begin{array}{lll}
    \xi_1&\cdots & \xi_n
      \end{array}\right],\quad
  C_2=\left[\begin{array}{lll}
    \eta_1&\ \cdots & \eta_n
  \end{array}\right],
\]
and let
\begin{equation}\label{Pick}
    P=\left[\frac{\eta_k^*\xi_j-\xi_k^*\eta_j}{\lambda_j-\bar\lambda_k}\right]_{j,k=1}^n
\end{equation}
be the unique solution of the Lyapunov equation
\begin{equation}\label{Lyap}
    PB_2-B_2^*P=C_2^*C_1-C_1^*C_2.
\end{equation}
Assume that $P$ is nonnegative and nondegenerate and that $\ran C_2=\dC^d$. Then the
data set $(B_1,B_2,C_1,C_2,P)$ satisfies the assumptions $(A1)-(A3)$. Consider the AIP
corresponding to this data set. Due to Lemma~\ref{lm:1.1} and Proposition~\ref{Funique}
every AIP solution is equivalent to a mvf $m(\cdot)\in N^{d\times d}$ and the mapping
$F:\cH\to\cH(m)$ in $(C1)$, $(C2)$ is uniquely defined by~\eqref{FG}. Therefore, the
conditions $(C1)$, $(C2)$  can be rewritten as
\begin{equation}\label{Tan4}
    (Fu)(\lambda):=\begin{bmatrix} I_d &
    -m(\lambda)\end{bmatrix}\begin{bmatrix} C_1\\ C_2 \end{bmatrix}(B_2-\lambda)^{-1}u\in\cH(m)\quad(u\in\dC^n);
\end{equation}
\begin{equation}\label{Tan5}
    \|Fu\|^2_{\cH(m)}\le(Pu,u)\quad(u\in\dC^n).
\end{equation}
We claim that the problem~\eqref{Tan1} is equivalent to the problem~\eqref{Tan4},
\eqref{Tan5}. Indeed, the condition $(C1)$ takes the form
\[
\left[\frac{\xi_j-m(\lambda)\eta_j}{\lambda-\lambda_j}\right]_{j=1}^nu\in\cH(m),\quad
u\in\dC^n,
\]
which is equivalent to $(B1)$. Moreover, if  $m(\cdot)$ has the integral representation
\[
m(\lambda)=\alpha+\beta\lambda+\int_\dR\left(\frac{1}{t-\lambda}-\frac{t}{1+t^2}\right)d\sigma(t),
\]
where $\alpha,\beta\in\dC^{d\times d}$, $\alpha=\alpha^*$, $\beta\ge 0$ and $\sigma(t)$
is a nondecreasing $d\times d$- valued mvf, then the vvf $(Fu)(\lambda)$ takes the form
\[
(Fu)(\lambda)=\left[\beta\eta_j+\int_\dR\frac{d\sigma(t)\eta_j}{(t-\lambda)(t-\lambda_j)}\right]_{j=1}^n
u.
\]
Due to~\cite[Theorem 2.5]{AG04} one obtains
\begin{equation}\label{Tan6}
\begin{split}
    \|Fu\|^2_{\cH(m)}
    &=u^*\left[\eta_k^*\left(\beta+\int_\dR\frac{d\sigma(t)}{(t-\lambda_j)(t-\bar\lambda_k)}\right)\eta_j\right]_{j,k=1}^n
u\\
&=u^*\left[\eta_k^*\frac{m(\lambda_j)-m(\lambda_k)^*}{\lambda_j-\bar\lambda_k}\eta_j\right]_{j,k=1}^nu.
\end{split}
\end{equation}
Thus, \eqref{Tan5} is implied by \eqref{Tan6} and \eqref{Tan1}.

More general bitangential interpolation problems in the classes of Nevanlinna pairs with
multiple points can be included in the AIP by using the data set $(B_1,B_2,C_1,C_2,P)$:
\begin{enumerate}
    \item
$ B_1=I_N$, $B_2=\mbox{diag }
    (J(\lambda_1)\,\dots\,J(\lambda_\ell) )$,
 $\lambda_j\in\cmr$ $(1\le j\le \ell)$, and $J(\lambda_j)$ is a Jordan cell,
corresponding to the eigenvalue $\lambda_j$
\[
J(\lambda_j)=\left[\begin{array}{llll}
    \lambda_j& 1     &      & \\
             &\ddots & \ddots & \\
             &       &\ddots & 1\\
             &       &       &\lambda_j
      \end{array}\right],\quad (1\le j\le \ell);
\]
of the order $n_j$,  $n=n_1+n_2+\dots+n_\ell$.
\item
$
   C_1=\left[\begin{array}{lll}
    \xi_1&\cdots & \xi_n
      \end{array}\right],\quad
  C_2=\left[\begin{array}{lll}
    \eta_1&\ \cdots & \eta_n
  \end{array}\right];
$
\item
$P$ is a nonnegative solution $P$ of~\eqref{Lyap}.
\end{enumerate}

If the set $\{\lambda_j\}_{j=1}^\ell$ contains symmetric points then the solution $P$ of
the Lyapunov equation~\eqref{Lyap} is not unique. Assume that there is a nonnegative
nondegenerate solution $P$ of~\eqref{Lyap}. Then the AIP corresponding to the data set
$(B_1,B_2,C_1,C_2,P)$ can be formulated as follows. Find a Nevanlinna pair
$\{\varphi,\psi\}$ such that:
\begin{equation}\label{Tan8}
    (Fu)(\lambda):=\begin{bmatrix} \varphi(\lambda) &
    -\psi(\lambda)\end{bmatrix}\begin{bmatrix} C_1\\ C_2 \end{bmatrix}(B_2-\lambda
)^{-1}u\in\cH(m)\quad(u\in\dC^n);
\end{equation}
\begin{equation}\label{Tan9}
    \|Fu\|^2_{\cH(m)}\le(Pu,u)\quad(u\in\dC^n).
\end{equation}
One can show that every solution $\{\varphi,\psi\}$ of the problem~\eqref{Tan8},
\eqref{Tan9} satisfies the Parseval equality
\[
 \|Fu\|^2_{\cH(m)}=(Pu,u)\quad(u\in\dC^n).
\]
Regular bitangential interpolation problems in the Schur and Nevanlinna classes were
studied in~\cite{Nud77}, \cite{BH83}, \cite{Dym89}, \cite{Kh1}, \cite{BGR},
\cite{ABDS93}, \cite{ABGR94}. Singular tangential and bitangential interpolation
problems considered in~\cite{Fed75},~\cite{Nud77},~\cite{Dym89}, \cite{Bru91},
\cite{BolDym98}, \cite{Dym01} can be included in the above consideration by imposing the
assumption (A2) on the data set.

\subsection{Hamburger moment problem.} Let $s_j\in\dC^{d\times d}$, $j\in\dN$
and let $S_n$ be the Hankel block matrix
\[
 S_n=(s_{i+j})_{i, j=0}^n.
\]
A $\dC^{d\times d}$- valued nondecreasing right continuous function $\sigma(t)$ is
called a solution of the Hamburger moment problem if
\begin{equation}\label{eq:6.1H}
\int t^j d\sigma(t)=s_j \quad (j\in\dN).
\end{equation}
It is known that the Hamburger moment problem~\eqref{eq:6.1H} is
solvable iff $S_n\geq 0$ for all $n\in\dN$. Due to
Hamburger-Nevanlinna theorem a function $\sigma(t)$ is a solution of
the Hamburger moment problem~\eqref{eq:6.1H} if and only if the
associated mvf
\[
m(\lambda)=\int_{-\infty}^\infty\frac{d\sigma(t)}{t-\lambda}
\]
has the following nontangential asymptotic at $\infty$
\begin{equation}\label{eq:6.2H}
    m(\lambda)\sim-\frac{s_0}{\lambda}-\frac{s_1}{\lambda^2}-\cdots-\frac{s_{2n}}{\lambda^{2n+1}}
    +O(\frac{1}{\lambda^{2n+2}}) \quad(\lambda\hat\to\infty)
\end{equation}
for every $n\in\dN$.

 Let $\cL=\dC^d$, let $\cX$ be the space of all vector polynomials
\begin{equation}\label{eq:6.0}
    h(X)=\sum_{j=0}^n h_jX^j,\quad h_j\in\cL,
\end{equation}
and let the nonnegative form $K(h,h)$ be defined by
\begin{equation}\label{eq:6.0K}
    K(h,h)=\sum_{j,k=0}^n (s_{j+k}h_j,h_k)_{\cL}.
\end{equation}
Assume that all the matrices are nondegenerate and consider the closure $\cH$ of the
space $\cX$ endowed with the inner product $K(\cdot,\cdot)$. Then the closure  $M$ of
the multiplication operator $M_0$ in $\cX$ is a symmetric operator in $\cH$. The moment
problem~\eqref{eq:6.1H} is called {\it indeterminate} if the defect numbers of $M$ are
equal to $d$. As was shown in~\cite{Berg02} the scalar moment problem ($d=1$) is
indeterminate if and only if there exists $\delta>0$ such that
\begin{equation}\label{eq:6.1h}
    S_n\ge \delta>0\quad\mbox{ for all }n\in\dN.
\end{equation}
Slight modification of the proof of this statement in~\cite{Berg02} shows that the
condition~\eqref{eq:6.1h} is also necessary and sufficient for the moment
problem~\eqref{eq:6.1H} to be  indeterminate for arbitrary $d$.

Let us consider the abstract interpolation problem in the class $N^{d\times d}$
corresponding to the indeterminate moment problem~\eqref{eq:6.1H}. Define the operators
$B_1,B_2:\cX\to\cX$ and $C_1,C_2:\cX\to\cL$ by the equalities
\begin{equation}\label{eq:6.3H}
        B_1h=\frac{h(X)-h_0}{X},\quad B_2h=h,\quad
        C_1h=\sum_{j=1}^n s_{j-1}h_j,\quad C_2h=-h(0).
\end{equation}
Then the data set $(B_1,B_2,C_1,C_2,K)$ satisfies the assumption
(A1). (A2) is clearly in force, since $\ker K=\{0\}$. Let us show
that (A3) is also in force.
\begin{proposition}\label{B12C12}
    The operators $B_1,C_1,C_2$ admit continuous extensions to
    the operators $\wt B_1\in[\cH]$, and $\wt C_1,\wt
    C_2\in[\cH,\cL]$.
\end{proposition}
\begin{proof}
Let $\wt B_1$ be the closure of the graph of the operator $B_1$.
Then
\[
\wt B_1^{-1}=\{\{h,Mh+u\}:\,h\in\dom M,\,u\in\cL\}.
\]
As was shown in~\cite{Kr49} $\rho(M,\cL)=\dC$ in the case of indeterminate moment
problem~\eqref{eq:6.1H}. In particular, $0\in\rho(M,\cL)$, that is
\[
\ran\wt B_1^{-1}=\ran M\dotplus\cL=\cH,\quad \ker\wt B_1^{-1}=\ran
M\cap\cL=\{0\}.
\]
Therefore $\wt B_1$ is the graph of a bounded  operator in $\cH$ for
which we will keep the same notation.

It follows from~\eqref{eq:6.1h} that for every polynomial $h\in\cX$
\[
\begin{split}
\|h\|_\cH^2&=K(h,h)=\sum_{j,k=0}^n h_k^*s_{j+k}h_j\\
&\ge\delta\sum_{j,k=0}^n \|h_j\|^2 \ge\delta\|h_0\|^2.
\end{split}
\]
Therefore,
\[
\|C_2h\|^2\le\frac{1}{\delta}\|h\|^2_\cH
\]
and, hence, the operator $C_2:\cX\subset\cH\to\cL$ is bounded. Let us note that the
boundedness of $C_2$ is implied also by the fact that $0\in\rho(M,\cL)$, since
$C_2h=\cP_{M,\cL}(0)h$, where $\cP_{M,\cL}(\lambda)$ is the skew projection onto $\cL$
in the decomposition $\cH\oplus\cL=\ran(M-\lambda)\dotplus \cL$.

The boundedness of $C_1:\cX\subset\cH\to\cL$ is implied by the
equality
\begin{equation}\label{C1}
    (C_1h,u)_\cL=K\left(\frac{h(X)-h(0)}{X},u\right)=K(B_1h,u).
\end{equation}
Indeed, it follows from~\eqref{C1} that
\[
\begin{split}
 |(C_1h,u)_\cL|&\le K(B_1h,B_1h)^{1/2}K(u,u)^{1/2}\\
 &=\|B_1h\|_\cH(s_0u,u)^{1/2}\\
 &\le\|B_1\|\|s_0^{1/2}\|\|h\|_\cH\|u\|_\cL
\end{split}
\]
and hence $C_1$ is bounded and
\[
\|C_1\|\le\|B_1\|\|s_0^{1/2}\|.
\]
\end{proof}
\begin{remark}
The definition \eqref{eq:6.3H} of $C_1$ can be rewritten as
\begin{equation}\label{C1A}
    C_1h=\wt h(0),
\end{equation}
where the adjacent polynomial $\wt h$ is defined by
\begin{equation}\label{Adj}
    (\wt
    h(\lambda),u)_\cL=K\left(\frac{h(X)-h(\lambda)}{X-\lambda},u\right),\quad
    u\in\cL.
\end{equation}
\end{remark}
Let us consider the $[\cH,\cL^2]$-valued operator function
\begin{equation}\label{eq:6.7H}
    G(\lambda)=\left[%
\begin{array}{cc}
  C_1 \\
  C_2 \\
\end{array}%
\right](I-\lambda B_1)^{-1},\quad\lambda\in\dC.
\end{equation}
Recall some useful formulas (see~\cite{Kh96})
\begin{equation}\label{eq:6.4H}
(I-\lambda B_1)^{-1}h=\frac{Xh(X)-\lambda h(\lambda)}{X-\lambda},\quad (h\in\cX).
\end{equation}
\begin{equation}\label{CG1}
    C_1(I-\lambda B_1)^{-1}h=\wt h(\lambda),\quad  C_2(I_m-\lambda B_1)^{-1}h=-
    h(\lambda).
\end{equation}
The corresponding abstract interpolation problem can be formulated
as follows. 
Find a Nevanlinna pair $\{\varphi,\psi\}\in\wt N(\dC^d)$, such that
\begin{equation}\label{eq:6.8H}
    Fh:=\begin{bmatrix} \varphi(\lambda) &
    -\psi(\lambda)\end{bmatrix}G(\lambda)h\in\cH(\varphi,\psi);
\end{equation}
\begin{equation}\label{eq:6.9H}
    \|Fh\|^2_{\cH(\varphi,\psi)}\le K(h,h)
\end{equation}
for all $h\in \cX$.

Since the operators $B_1$, $B_2=I$ satisfy the assumption (U) the mapping
$F:\cX\to\cH(\varphi,\psi)$ corresponding to the solution $\{\varphi,\psi\}$ of the AIP
is uniquely defined (see Proposition~\ref{FG}). Moreover, since $\ran C_2=\cL$ it
follows from Lemma~\ref{lm:1.1} that any solution $\{\varphi,\psi\}$ of the AIP is
equivalent to a pair $\{I_d, m(\lambda)\}$, where $m\in N^{d\times d}$.
\begin{theorem}\label{HamAIP}
    Let $m$ be a solution of the $AIP(B_1,B_2,C_1,C_2,K)$,
    which assumes the integral representation~\eqref{eq:6.2},
    and let $F:\cX\to\cH(\varphi,\psi)$ be the mapping corresponding to the
    solution $m$ and given by the formula~\eqref{eq:6.8H}. Then:
\begin{enumerate}
\item $\sigma$ is a solution of the Hamburger moment
problem~\eqref{eq:6.1H};
\item for every polynomial $h\in\cX$ one has
\begin{equation}\label{F1}
    (Fh)(\lambda)=\int_{-\infty}^\infty\frac{d\sigma(t)h(t)}{t-\lambda},
\end{equation}
\begin{equation}\label{F2}
    \|Fh\|^2_{\cH(m)}=K(h,h).
\end{equation}
\end{enumerate}
\end{theorem}
\begin{proof}
For a monic polynomial $h=uX^j$ $(u\in\cL)$ one obtains
from~\eqref{eq:6.3H}-\eqref{eq:6.7H}
\[
\begin{split}
(Gh)(\lambda)&=\left[%
\begin{array}{cc}
  C_1 \\
  C_2 \\
\end{array}%
\right](X^j+\lambda X^{j-1}+\dots +\lambda^j)u\\
&=\left[%
\begin{array}{cc}
  (\lambda^{j-1}s_0+\dots+\lambda s_{j-2}+s_{j-1})u\\
  -\lambda^ju \\
\end{array}%
\right]
\end{split}
\]
and hence
\begin{equation}\label{eq:6.10}
 \begin{split}
(Fh)(\lambda)&=\left[%
\begin{array}{cc}
  I_d & -m(\lambda)
\end{array}\right]%
G(\lambda)h \\
&=
  (\lambda^{j}m(\lambda)+\lambda^{j-1} s_{0}+\dots +s_{j-1})u\\
&=\int_{-\infty}^\infty\frac{t^j}{t-\lambda}d\sigma(t)u.
\end{split}
\end{equation}
When $j=0$ it follows from \eqref{eq:6.10} and Lemma~\ref{lm:2.} that $m(\lambda)=O(1)$
if $\lambda\wh\to\infty$. Setting $j=1$ one derives from \eqref{eq:6.10} that
\[
\lambda m(\lambda)+s_0=O(1)\quad (\lambda\wh\to\infty).
\]
Therefore, $m(\lambda)=O(\frac{1}{\lambda})$ and applying
Lemma~\ref{lm:2.}, (ii) gives
\[
m(\lambda)+\frac{s_0}{\lambda}=O(\frac{1}{\lambda^2}).
\]
And similarly, for $h=u X^n$ one obtains from (C1) $(\lambda^n
m(\lambda)+s_0\lambda^{n-1}+\dots+s_{n-1})u\in\cH(m)$, or by Lemma~\ref{lm:2.}, (ii)
\[
m(\lambda)+\frac{s_0}{\lambda}+\dots+\frac{s_{n-1}}{\lambda^n}=O(\frac{1}{\lambda^{n+1}})
\]
for arbitrary $n\in\dN$. In view of Hamburger-Nevanlinna theorem
this implies that $\sigma$ is a solution of the Hamburger moment
problem~\eqref{eq:6.1H}.

For arbitrary polynomial $h=\sum_{j=0}^nu_jX^j\in\cX$ the
formula~\eqref{eq:6.10} can be rewritten as
\[
(Fh)(\lambda)=\int_{-\infty}^\infty\frac{d\sigma(t)h(t)}{t-\lambda}.
\]
 Applying the formula for
the inner product in $\cH(m)$ (see~\cite[Theorem 2.5]{AG04}) one obtains
\[
\|Fh\|^2_{\cH(m)}=\int_{-\infty}^\infty (d\sigma(t)h(t),h(t))
=\sum_{j,k=0}^n(s_{j+k}u_j,u_k)_\cL=K(h,h).
\]
This proves~\eqref{F2}.

Conversely, let $\sigma$ be a solution of the Hamburger moment problem~\eqref{eq:6.1H}.
Then it follows from \eqref{F1} and Theorem 2.5 from~~\cite{AG04} that $Fh\in\cH(m)$ for
arbitrary polynomial $h\in\cX$. This proves (C1). (C2) is implied by the
equality~\eqref{F2}.
\end{proof}

To calculate the $\cL$-resolvent matrix let us introduce a system of
matrix polynomials $\{P_n(\lambda)\}_{n=0}^\infty$ orthogonal with
respect to the form $K$:
\[
K(P_ju,P_kv)=v^*u\delta_{jk},\quad j,k\in\dN
\]
and a system of adjacent polynomials $\{\wt
P_k(\lambda)\}_{n=0}^\infty$
\[
v^*\wt
P_k(\lambda)u=K\left(\frac{P_k(t)-P_k(\lambda)}{t-\lambda}u,v\right)\quad
k\in\dN.
\]
\begin{proposition}
For every $u\in\cL$ the following formulas hold
\begin{equation}\label{C1*}
    C_1^*u=\sum_{k=1}^\infty P_k(t)\wt P_k(0)^*u,\quad
    C_2^*u=-\sum_{k=0}^\infty P_k(t) P_k(0)^*u.
\end{equation}
\end{proposition}
\begin{proof}
Indeed, for every $j\in\dN\cup\{0\}$, $u,v\in\cL$ one obtains
from~\eqref{eq:6.3H}, \eqref{C1A}
\[
\begin{split}
(C_1^*u,P_jv)_\cH&=(u,C_1P_jv)_\cL=(u,\wt P_j(0)v)_\cL\\
&=(\wt P_j(0)^*u,v)_\cH=(\sum_{k=1}^\infty P_k(t)\wt
P_k(0)^*u,P_jv)_\cH,
\end{split}
\]
\[
\begin{split}
(C_2^*u,P_jv)_\cH&=(u,C_2P_jv)_\cL=-(u, P_j(0)v)_\cL\\
&=-( P_j(0)^*u,v)_\cH=-(\sum_{k=1}^\infty P_k(t)P_k(0)^*u,P_jv)_\cH,
\end{split}
\]
\end{proof}
Applying the formulas~\eqref{eq:4.24}, \eqref{CG1} and \eqref{C1*} one obtains the
resolvent matrix $\Theta(\lambda)$:
\[
\begin{split}
\theta_{11}(\lambda)u&=u-\lambda \wt C_1(I-\lambda \wt B_1)^{-1}(\sum_{k=0}^\infty
P_k(t)
P_k(0)^*u)=(I+\lambda\sum_{k=0}^\infty \wt P_k(\lambda) P_k(0)^*)u,\\
\theta_{12}(\lambda)u&=\lambda \wt C_1(I-\lambda \wt B_1)^{-1}(\sum_{k=0}^\infty P_k(t)
\wt P_k(0)^*u)=\lambda\sum_{k=0}^\infty \wt P_k(\lambda) \wt P_k(0)^*u,\\
\theta_{21}(\lambda)u&=-\lambda \wt C_2(I-\lambda \wt B_1)^{-1}(\sum_{k=0}^\infty P_k(t)
P_k(0)^*u)=-\lambda\sum_{k=0}^\infty  P_k(\lambda) P_k(0)^*u,\\
\theta_{22}(\lambda)u&=u+\lambda \wt C_2(I-\lambda \wt B_1)^{-1}(\sum_{k=1}^\infty
P_k(t) \wt P_k(0)^*u)=(I-\lambda\sum_{k=1}^\infty P_k(\lambda)\wt P_k(0)^*)u.
\end{split}
\]
Application of general result in Corollary~\ref{mdescr} gives the well known description
of solutions of the moment problem~\eqref{eq:6.1H}
\begin{equation}\label{MPdescr}
\int_\dR\frac{d\sigma(t)}{t-\lambda}=(\theta_{11}(\lambda)q(\lambda)+\theta_{12}(\lambda)p(\lambda))
(\theta_{21}(\lambda)q(\lambda)+\theta_{22}(\lambda)p(\lambda))^{-1},
\end{equation}
when the pair $\{p,q\}$ ranges over the class $N(\dC^d)$.
\subsection{Truncated Hamburger moment problem.}
Let $s_0$, $s_1,\ldots, s_{ 2n}\in\dC^{d\times d}$. A $\dC^{d\times d}$ -- valued
nondecreasing right continuous function $\sigma(\lambda)$ is called a solution of the
truncated Hamburger moment problem if
\begin{equation}\label{eq:6.1}
\int t^j d\sigma(t)=s_j \quad (j=0, 1, \ldots, 2n-1)
\end{equation}
\begin{equation}\label{eq:6.2}
 \int t^{2n} d \sigma(t)\leq s_{2n}.
\end{equation}
It is known that the problem~\eqref{eq:6.1}--~\eqref{eq:6.2} is solvable if and only if
$S_n\geq 0$. A solution $\sigma$ of the problem~\eqref{eq:6.1}--~\eqref{eq:6.2} is
called "exact" if
\begin{equation}\label{eq:6.3}
\int t^{2n} d\sigma(t)=s_{2n}.
\end{equation}
Singular truncated Hamburger moment problem has been studied in~\cite{CurtoF91},
\cite{Bol96}, \cite{AdTk00}.
\begin{theorem}\label{th:6.1}{\rm(\cite{CurtoF91}, \cite{Bol96})} Let
$S_n=(s_{i+j})_{i,j=0}^n$ -- be a nonnegative block Hankel matrix $(s_j\in\dC^{d\times
d})$. The following assertions are equivalent:
\begin{enumerate}
\item[ 1)] The problem~\eqref{eq:6.1}--~\eqref{eq:6.2} has an
"exact" solution;
\item[ 2)] $\ran\left[%
\begin{array}{c}
  s_{n+1} \\
  \vdots \\
  s_{2n} \\
\end{array}%
\right]\subseteq\ran S_{n-1}$; \item[ 3)] $S_n$ admits a nonnegative block Hankel
extension $\wt S_n$.
\end{enumerate}
\end{theorem}
The equivalence $(1)\Leftrightarrow(2)$ and $(2)\Leftrightarrow(3)$ were proved
in~\cite{CurtoF91} and~\cite{Bol96}, respectively.

Moreover, it was shown in~\cite{Bol96} that if the conditions 1)--3) in
Theorem~\ref{th:6.1} fail to hold then one can replace the right lower block $s_{2n}$ in
the matrix $S_n$ by ${s}'_{2n}$ in such a way that the new matrix
${S_n}'=({s'_{i+j}})_{i, j=0}^n$ with ${s_i}'=s_i$ for $i\leq 2n-1$ satisfies 1)-3) in
Theorem~\ref{th:6.1} and the sets $\cZ(S_n)$ and $\cZ({S'_n})$ coincide.

In what follows it is supposed that $S_n$ satisfies the assumptions 1)-3) of
Theorem~\ref{th:6.1}. We will need also the following statement from~\cite{Bol96}.
\begin{lemma}\label{Pseudo}
    Let a block Hankel matrix $S_n=(s_{i+j})_{i,j=0}^n$ satisfy the assumptions of
    Theorem~\ref{th:6.1} and let the matrix $T\in\dC^{N\times N}$ $(N=(n+1)d)$ be given by
\[
T=\begin{bmatrix}
0_d & I_d & & \\
    & \ddots &\ddots & \\
    &     & \ddots & I_d \\
    &   &   & 0_d
\end{bmatrix}.
\]
Then there exists a matrix $X=X^*\in\dC^{N\times N}$ of $\rank X=\rank S_n$  such that
\begin{equation}\label{XQ}
    XS_nX=X,\,\,S_nXS_n=S_n,\,\, T\ran X\subseteq\ran X.
\end{equation}
\end{lemma}

Let $\cX$ be the space of vector polynomials $h(X)$ of the form~\eqref{eq:6.0} of formal
degree  $n$ and let the form $K(\cdot,\cdot)$ be given by~\eqref{eq:6.0K}. Define the
operators $B_1,B_2:\cX\to\cX$ and $C_1,C_2:\cX\to\cL$ by~\eqref{eq:6.3H}. Then the data
set $(B_1,B_2,C_1,C_2,K)$ satisfies the assumption (A1). Choosing the basis $1, X,\dots,
X^n$ in $\cX$ one can identify $\cX$ with $\dC^{N}$ and then the form $K(\cdot,\cdot)$
is given by
\[
K(h,g)=(S_nh,g),\quad h,g\in\dC^N,\quad N={(n+1)d}.
\]
The operators $B_1,B_2$ and $C_1,C_2$ can be identified with their matrix
representations in this basis
\begin{equation}\label{eq:6.5}
B_1=T, \quad B_2=I_{N},
\end{equation}
\begin{equation}\label{eq:6.4}
C_1=\begin{bmatrix}0 & s_0 &\ldots & s_{n-1}\end{bmatrix}, \quad C_2=\begin{bmatrix}-I_d
&0 &\ldots & 0\end{bmatrix}.
\end{equation}
Then $B_2-\lambda B_1=I_{N}-\lambda T$ is invertible for all
$\lambda\in\dC\setminus\{0\}$
\begin{equation}
G(\lambda)=
\begin{bmatrix}
C_1\\ C_2
\end{bmatrix}(I_{N}-\lambda T)^{-1}=
\begin{bmatrix}
C_1\\ C_2
\end{bmatrix}
\begin{bmatrix}
I_d &\dots & \lambda^n I_d\\
    & \ddots     & \vdots \\
    &            & I_d
\end{bmatrix}.
\end{equation}
Since the operators $B_1=T$, $B_2=I_N$ satisfy the assumption (U) and $\ran C_2=\cL$ the
mapping $F:\cX\to\cH(\varphi,\psi)$ corresponding to the solution $\{\varphi,\psi\}$ of
the  $AIP(B_1,B_2,C_1,C_2,K)$ is uniquely defined and any solution $\{\varphi,\psi\}$ of
the  AIP is equivalent to a pair $\{I_d, m(\lambda)\}$, where $m\in N^{d\times d}$. The
corresponding AIP can be formulated as follows:

Find a Nevanlinna mvf $m\in N^{d\times d}$ such that:
\begin{enumerate}
\item[(C1)] $F h=\begin{bmatrix} I_d
&-m(\lambda)\end{bmatrix}G(\lambda)h\in \cH(m)$ for all $h\in\cX$;
\item[(C2)] $\|Fh\|_{\cH(m)}^2\leq(S_n h,h) $ for all $h\in\cX$.
\end{enumerate}

\begin{proposition}\label{pr:6.2}
Let $m$ be a solution of the $AIP(B_1,B_2,C_1,C_2,S_n)$. Then $m$ admits the integral
representation
\begin{equation}\label{eq:6.11}
m(\lambda)=\int_{-\infty}^\infty\frac{d\sigma(t)}{t-\lambda}
\end{equation}
where $\sigma\in \cZ(S_n)$. Conversely, if $\sigma\in \cZ(S_n)$, then $m$ is a solution
of the AIP.
\end{proposition}
\begin{proof}
{\it Necessity.} The same arguments as in the proof of Theorem~\ref{HamAIP} show that
(C1) implies
\begin{equation}\label{eq:6.12}
    m(\lambda)+\frac{s_0}{\lambda}+\dots+\frac{s_{n-1}}{\lambda^n}=O(\frac{1}{\lambda^{n+1}})\quad
(\lambda\wh\to\infty).
\end{equation}
Let $m(\lambda)$ have the integral representation \eqref{eq:6.11} and let us set
\begin{equation}\label{eq:6.1m}
s_j^{(m)}=\int_\dR t^j d\sigma(t) \quad (j=0, 1, \ldots, 2n).
\end{equation}
It follows from~\eqref{eq:6.12} that
\begin{equation}\label{eq:6.131}
    s_j^{(m)}=s_j
 \mbox{ for } j=0,1,\dots,n-1.
\end{equation}
The rest of the equalities~\eqref{eq:6.1} and the inequality~\eqref{eq:6.2} are implied
by the condition (C2).

Let us show that for every polynomial $ h(X)=\sum_{j=0}^n u_jX^j,\quad u_j\in\dC^{d}$,
the following equality holds
\begin{equation}\label{eq:6.14}
    \|F h\|^2_{\cH(m)}=\sum_{j,k=0}^n u_k^*s_{j+k}^{(m)}u_j.
\end{equation}
Indeed, it follows from~\eqref{eq:6.10} that
\begin{equation}\label{eq:6.141}
\begin{split}
(F h)(\lambda)&=
\begin{bmatrix}
m(\lambda)&\lambda m(\lambda)+s_0&\dots &\lambda^n
m(\lambda)+\lambda^{n-1}s_0+\dots+s_{n-1}
\end{bmatrix}u\\
&=\begin{bmatrix} {\displaystyle\int_\dR\frac{d\sigma(t)}{t-\lambda}}&
{\displaystyle\int_\dR\frac{td\sigma(t)}{t-\lambda}}&\dots &
{\displaystyle\int_\dR\frac{t^nd\sigma(t)}{t-\lambda}}
\end{bmatrix}u\in\cH(m),
\end{split}
\end{equation}
where $u=\col(u_0,u_1,\ldots,u_n)\in\dC^{(n+1)d}$. Due to~\cite[Theorem 2.5]{AG04}
\begin{equation}\label{eq:6.15}
    \|F h\|^2_{\cH(m)}=u^*S_n^{(m)}u,
\end{equation}
where
\begin{equation}\label{eq:6.16}
    S_n^{(m)}=\left[s_{i+j}^{(m)}\right]_{i,j=0}^n.
\end{equation}
Then the inequality $S_n^{(m)}\le S_n$ in (C2) and~\eqref{eq:6.131} imply
that~\eqref{eq:6.1} and~\eqref{eq:6.2} hold.

{\it Sufficiency.} Let $\sigma\in\cZ(S_n)$ and let $m$ be defined by~\eqref{eq:6.11}.
Then it follows from~\eqref{eq:6.141} that (C1) holds.

The condition (C2) is implied by~\eqref{eq:6.15}, \eqref{eq:6.1} and \eqref{eq:6.2},
since
\[
 \|Fh\|^2_{\cH(m)}=u^*S_n^{(m)}u\le u^*S_nu,\quad u\in\dC^{N}.
\]
\end{proof}
In the regular case (when $\det S_n\ne 0$) the solution matrix $\Theta(\lambda)$ can be
calculated by~\eqref{eq:Theta}. Since $C_1^+=S_n^{-1}C_1^*$ and $C_2^+=S_n^{-1}C_2^*$
one obtains from~\eqref{eq:Theta} that
\[
\Theta(\lambda)=I_{2d}-\lambda\left[%
\begin{array}{c}
  C_1 \\
  C_2 \\
\end{array}%
\right](I_\cH-\lambda T)^{-1}S_n^{-1}\begin{bmatrix} C_2^* & - C_1^*\end{bmatrix}.
\]
Then by Corollary~\ref{mdescr} the formula \eqref{MPdescr} establishes the 1-1
correspondence between the set of all solutions $\sigma\in\cZ(S_n)$ of the truncated
moment problem~\eqref{eq:6.1}-\eqref{eq:6.2} and the set of all equivalence classes of
Nevanlinna pairs $\{p,q\}\in\wt N^{d\times d}$.

In the singular case ($\det S_n= 0$) let us consider the matrix $X\in\dC^{N\times N}$
which satisfies~\eqref{XQ}. Then the decomposition
\[
\cX=\ran X\dotplus\ker S_n
\]
satisfies the assumptions (A2), (A3), since $T\ran X\subseteq \ran X$, and the solution
matrix $\Theta(\lambda)$ takes the form~\eqref{eq:Theta}. Now, let us calculate the
operators $C_1^+,C_2^+:\cL\to\cH$. For arbitrary $h=Xg\in\ran X$, $u\in\cL$ and $j=1,2$
one obtains
\[
(C_jXg,u)_{\dC^d}=(Xg,C_j^*u)_{\dC^N}=(Xg,S_nXC_j^*u)_{\dC^N}=(Xg,XC_j^*u)_{\cH}.
\]
Therefore, $C_1^+=XC_1^*,\,\,C_2^+=XC_2^*$, and the mvf $\Theta(\lambda)$ takes the form
\begin{equation}\label{eq:4.333}
\Theta(\lambda)=\left(I_{\cL\oplus\cL}-\lambda\left[%
\begin{array}{c}
   C_1 \\
   C_2 \\
\end{array}%
\right](I_\cH-\lambda T)^{-1}X\begin{bmatrix} C_2^* & -C_1^*\end{bmatrix}\right)V,
\end{equation}
where $V\in\dC^{2d\times 2d}$ is a unitary matrix such that
\[
V(\{0\}\times \dC^\nu)\!=\begin{bmatrix}C_1\\-C_2\end{bmatrix}\ker
S_n=\begin{bmatrix}0&s_1&\dots &s_{n-1}\\I_d &0 &\dots & 0\end{bmatrix}\ker S_n
\]
and
\[
\nu=\dim\begin{bmatrix}0&s_1&\dots &s_{n-1}\\-I_d &0 &\dots & 0\end{bmatrix}\ker S_n.
\]
Then by Corollary~\ref{mdescr} the formula \eqref{MPdescr} establishes the 1-1
correspondence between the set of all solutions $\sigma\in\cZ(S_n)$ of the truncated
moment problem~\eqref{eq:6.1}-\eqref{eq:6.2} and the set of all equivalence classes of
Nevanlinna pairs $\{p,q\}\in\wt N^{d\times d}$ of the form~\eqref{pq}.

\end{document}